\documentclass[11pt]{aart}

\usepackage[letterpaper, hmargin=1in, top=1.1in, bottom=1.3in, footskip=0.7in]{geometry}

\usepackage{titlesec}
\titleformat{\section}[block]{\filcenter\normalfont\bfseries\large}{\thesection.}{.5em}{}\titlespacing*{\section}{0pt}{2\baselineskip}{1\baselineskip}
\titleformat{\subsection}[runin]{\normalfont\bfseries}{\thesubsection.}{.4em}{}[.]\titlespacing{\subsection}{0pt}{2ex plus .1ex minus .2ex}{.8em}
\titleformat{\subsubsection}[runin]{\normalfont\itshape}{\thesubsubsection.}{.3em}{}[.]\titlespacing{\subsubsection}{0pt}{1ex plus .1ex minus .2ex}{.5em}
\titleformat{\paragraph}[runin]{\normalfont\itshape}{\theparagraph.}{.3em}{}[.]\titlespacing{\paragraph}{0pt}{1ex plus .1ex minus .2ex}{.5em}

\usepackage{amsmath,bbm,color,amssymb,subfigure,graphicx,amsthm}

\definecolor{vdarkred}{rgb}{0.6,0,0.2}
\definecolor{vdarkblue}{rgb}{0,0.2,0.6}
\usepackage[pdftex, colorlinks, linkcolor=vdarkblue,citecolor=vdarkred]{hyperref}
\usepackage[small]{caption}

\DeclareMathOperator{\real}{Re}

\newcommand{\ld}{\ldots}
\newcommand{\beg}{\begin}
\newcommand{\en}{\end}

\newcommand{\trm}{\textrm}
 
\newcommand{\bgt}{\begin{itemize}}
\newcommand{\ent}{\end{itemize}}

\newcommand{\op}{\operatorname}
\newcommand{\eqre}{\eqref}
\newcommand{\la}{\label}

\newcommand{\si}{\sigma}

 \newcommand{\bgn}{\begin{enumerate}}
\newcommand{\enn}{\end{enumerate}}

\newcommand{\Tr}{\operatorname{Tr}}

\newcommand{\E}{\mathbb{E}}

\newcommand{\R}{\mathbb{R}}
\newcommand{\C}{\mathbb{C}}

\newcommand{\pro}{probability }

\newcommand{\f}{\frac}
\newcommand{\ff}{\frac{1}}

\newcommand{\st}{such that }

\newcommand{\lam}{\lambda}

\newcommand{\ti}{\times}

\newcommand{\eps}{\varepsilon}

\newcommand{\bck}{\backslash}

\newcommand{\ovl}{\overline}

\newcommand{\bbm}{\begin{bmatrix}}
\newcommand{\ebm}{\end{bmatrix}}
\newcommand{\bes}{\begin{equation*}}
\newcommand{\ees}{\end{equation*}}
\newcommand{\be}{\begin{equation}}
\newcommand{\ee}{\end{equation}}

\newcommand{\bpm}{\begin{pmatrix}}
\newcommand{\epm}{\end{pmatrix}}

\newcommand{\cd}{\cdots}
\newcommand{\wt}{\widetilde}

\newcommand{\bpr}{\beg{proof}}
\newcommand{\epr}{\en{proof}}

\newcommand{\del}{\delta}

\newcommand{\ka}{\kappa}

\newcommand*{\deq}{\mathrel{\vcenter{\baselineskip0.65ex \lineskiplimit0pt \hbox{.}\hbox{.}}}=}

\newcommand{\PAR}[1]{{{\left(#1\right)}}} 
\newcommand{\AND}{{\quad \trm{ and } \quad }}
\newtheorem{Th}{Theorem}[section]

\newtheorem{assum}[Th]{Assumption}

\newtheorem{proposition}[Th]{Proposition} 

\newtheorem{lemma}[Th]{Lemma}

\newtheorem{cor}[Th]{Corollary}

\theoremstyle{definition}

\newtheorem{rmk}[Th]{Remark}

\newtheorem{Def}[Th]{Definition}

\definecolor{darkblue}{rgb}{0,0.3,0.9}
\definecolor{darkred}{rgb}{0.9,0,0.3}

\newcommand{\bs}[1]{\boldsymbol{\mathrm{#1}}} %bold
\newcommand{\bb}{\mathbb} %blackboard bold
\renewcommand{\cal}{\mathcal}

\newcommand{\ol}[1]{\overline{#1} \!\,} %overline

\newcommand{\me}{\mathrm{e}}

\newcommand{\col}{\mathrel{\vcenter{\baselineskip0.75ex \lineskiplimit0pt \hbox{.}\hbox{.}}}}
\renewcommand{\leq}{\leqslant}
\renewcommand{\geq}{\geqslant}
\renewcommand{\le}{\leqslant}
\renewcommand{\ge}{\geqslant}
\renewcommand{\epsilon}{\varepsilon}

\newcommand{\ind}[1]{\bs 1_{#1}}
\newcommand{\pb}[1]{\bigl({#1}\bigr)}

\newcommand{\pbb}[1]{\biggl({#1}\biggr)}

\newcommand{\qbb}[1]{\biggl[{#1}\biggr]}

\newcommand{\h}[1]{\{{#1}\}}
\newcommand{\hb}[1]{\bigl\{{#1}\bigr\}}

\newcommand{\hbb}[1]{\biggl\{{#1}\biggr\}}

\newcommand{\abs}[1]{\lvert #1 \rvert}
\newcommand{\absb}[1]{\bigl\lvert #1 \bigr\rvert}

\newcommand{\norm}[1]{\lVert #1 \rVert}
\newcommand{\normb}[1]{\bigl\lVert #1 \bigr\rVert}
\newcommand{\normB}[1]{\Bigl\lVert #1 \Bigr\rVert}

\DeclareMathOperator{\diag}{diag}
\DeclareMathOperator{\tr}{Tr}

\numberwithin{equation}{section}

 \allowdisplaybreaks

\title{Spectral radii of sparse random matrices}
\author{Florent Benaych-Georges \and Charles Bordenave \and Antti Knowles} 

\date{September 19, 2019}

\begin{document}
\maketitle

\begin{abstract}
We establish bounds on the spectral radii for a large class of sparse random matrices, which includes the adjacency matrices of inhomogeneous Erd\H{o}s-R\'enyi graphs. Our error bounds are sharp for a large class of sparse random matrices. In particular, for the Erd\H{o}s-R\'enyi graph $G(n,d/n)$, our results imply that the smallest and second-largest eigenvalues of the adjacency matrix converge to the edges of the support of the asymptotic eigenvalue distribution provided that $d / \log n \to \infty$. Together with the companion paper \cite{bbkER}, where we analyse the extreme eigenvalues in the complementary regime $d / \log n \to 0$, this establishes a crossover in the behaviour of the extreme eigenvalues at $d \asymp \log n$. Our results also apply to non-Hermitian sparse random matrices, corresponding to adjacency matrices of directed graphs.
The proof combines (i) a new inequality between the spectral radius of a matrix and the spectral radius of its nonbacktracking version together with (ii) a new application of the method of moments for nonbacktracking matrices. 
\end{abstract}

\section{Introduction}

The goal of the present paper is to obtain bounds on the spectral radius of a sparse random matrix. Sparse random matrices arise naturally as adjacency matrices of random graphs. In spectral graph theory, obtaining precise bounds on the locations of the extreme eigenvalues, in particular on the spectral gap, is of fundamental importance and has attracted much attention in the past thirty years. See for instance \cite{Chu,HLW06,Alo98} for reviews.

The problem of estimating the spectral radius of a random matrix has a long history, starting with the seminal work of F{\"u}redi and Koml{\'o}s \cite{KF}, subsequently improved by Vu \cite{VuLargestEig}.
For the simple case of the Erd\H{o}s-R\'enyi graph $G(n,d/n)$, where each edge of the complete graph on $n$ vertices is kept with probability $d/n$ independently of the others, it is shown in \cite{KF,VuLargestEig} that the smallest and second-largest eigenvalues of the adjacency matrix converge to the edges of the support of the asymptotic eigenvalue distribution provided that $d / (\log n)^4 \to \infty$. In this paper we derive quantitative bounds for the extreme eigenvalues, which are sharp for a large class of sparse random matrices, including $G(n,d/n)$ for small enough expected degree $d$. An immediate corollary of these bounds is that the extreme eigenvalues converge to the edges of the support of the asymptotic eigenvalue distribution provided that $d / \log n \to \infty$.
Together with the companion paper \cite{bbkER}, where we analyse the extreme eigenvalues in the complementary regime $d / \log n \to 0$, this establishes a crossover in the behaviour of the extreme eigenvalues of the Erd\H{o}s-R\'enyi graph. The location $d \asymp \log n$ of this crossover for the extreme eigenvalues parallels the well-known crossover from connected to disconnected graphs. 

Our results hold for a very general class of sparse random matrices with independent mean-zero entries. In particular, we also obtain bounds for the extreme eigenvalues of inhomogeneous Erd\H{o}s-R\'enyi graphs and stochastic block models. They hold for Hermitian random matrices, corresponding to adjacency matrices of undirected random graphs, as well as for non-Hermitian random matrices with independent entries, corresponding to adjacency matrices of directed random graphs.

 In their seminal work,  F{\"u}redi and Koml{\'o}s \cite{KF} bound the spectral radius of a random matrix by estimating  the expected trace of  high powers of this matrix. The technical challenge of this strategy is to derive a sharp upper bound on the number of simple walks with given combinatorial properties. The novelty of our approach lies in a new use of the nonbacktracking matrix. Nonbacktracking matrices and combinatorial estimates on nonbacktracking walks have proved to be powerful tools to study largest eigenvalues of some random matrices, see \cite{MR1137767,MR2377835,bordenaveCAT}.  Indeed, counting non-backtracking walks turns out to be significantly simpler than counting  simple walks. Moreover, on regular graphs, the nonbacktracking walks are counted by polynomials of the adjacency matrix, and the Ihara-Bass formula  (see e.g.\ \cite{MR1749978}) implies a simple relation between the spectrum of the adjacency matrix and the nonbacktracking matrix. In this paper, we extend this strategy beyond regular graphs.  Let $H$ denote the sparse random matrix we are interested in, and $B$ its associated nonbacktracking matrix (defined in Definition \ref{def:B} below). Our proof consists of two main steps: a deterministic step (i) and a probabilistic step (ii).
\begin{enumerate}
\item
The first step is an estimate of the spectral radius of $H$ in terms of the spectral radius of $B$, summarized in Theorem \ref{thm:From_NBmatrix_to_matrix}. This step also requires bounds on the $\ell^2 \to \ell^\infty$ and $\ell^1 \to \ell^\infty$ norms of $H$, which are typically very easy to obtain by elementary concentration results. The main algebraic tool behind this step is an Ihara-Bass-type formula given in Lemma \ref{le:IB} below. Using this lemma we obtain an estimate relating the eigenvalues of $B$ and $H$, Proposition \ref{prop:AIB} below, from which Theorem \ref{thm:From_NBmatrix_to_matrix} easily follows.
\item
The second step is an estimate of the spectral radius of $B$, summarized in Theorem \ref{thm:rhoB} below. Our starting point is the classical F{\"u}redi-Koml{\'o}s-approach of estimating $\E \tr B^\ell B^{*\ell}$ for $\ell \gg \log n$, which may be analysed by counting walks on multigraphs. Our main result here is Proposition \ref{prop:BB_bound}. The argument is based on a new soft combinatorial argument which revisits a reduction of walks introduced in Friedman \cite{MR1137767}. 
\end{enumerate}

Our main result for Hermitian sparse matrices, Theorem \ref{thm:norm_H}, follows from Theorems \ref{thm:From_NBmatrix_to_matrix} and \ref{thm:rhoB} combined with classical concentration estimates. The proof of our main result for non-Hermitian sparse matrices, Theorem \ref{thm:norm_hH}, is proved by an argument analogous to that used to prove Theorem \ref{thm:rhoB}; the key observation here is that the independence of the entries of $H$ automatically results in a nonbacktracking constraint in the graphical analysis of $\E \tr H^\ell H^{*\ell}$.

In \cite{BandeiraVanHandel}, an approach to estimating spectral radii of random matrices is developed by comparison of the spectral radii of general inhomogeneous random matrices to those of corresponding homogeneous random matrices.
After our preprint appeared online, this approach was significantly extended in \cite{HLY17} to cover, among other things, sparse random matrices.
For $G(n,d/n)$ in the regime $d / \log n \to \infty$, both our approach and that of \cite{BandeiraVanHandel, HLY17} yield bounds that are sharp to leading order, although our bound on the subleading error is better than that of \cite{BandeiraVanHandel, HLY17}. In the complementary very sparse regime $d / \log n \to 0$ our bounds are sharp up to a universal constant while those of \cite{BandeiraVanHandel, HLY17} are not. See Remark \ref{rem:very sparse regime} below for more details.

We conclude this introduction with a summary of the main contributions of our paper.
\begin{enumerate}
\item
We obtain bounds on the spectral radius of a large class of sparse random matrices. For $G(n,d/n)$, our leading order bound is sharp in the regime $d / \log n \to \infty$ and sharp up to a universal constant in the regime $d / \log n \to 0$.
\item
We establish optimal bounds at leading order for the spectral radius of sparse non-Hermitian matrices  that are of mean-field type, which means that any $\abs{H_{ij}}$ is a.s.\ much smaller than the $\ell^2$ row norm $\sqrt{\max_i \sum_{j} \E \abs{H_{ij}}^2}$. 
\item
We establish optimal bounds at leading order for the spectral radius of the nonbacktracking matrix in the regime where the ratio of the $\ell^2 \to \ell^\infty$ and $\ell^1 \to \ell^\infty$ norms grows. Such bounds are of some independent interest notably in view of recent applications of non-backtracking matrices in community detection; see \cite{spec_red}. Moreover, our estimate on the spectral radius of the nonbacktracking matrix is a key ingredient in the subsequent analysis \cite{ADK} of the crossover at $d \asymp \log n$.
\item
We establish a very simple and general relationship between the spectral radii of a Hermitian matrix and its nonbacktracking version, which we believe can be of some use in other contexts as well.
\item
Our proof is simple and fully self-contained, and in particular does not need a priori bounds on the spectral radius of a reference matrix.
\end{enumerate}

\paragraph{Acknowledgements}
We are grateful to Sasha Sodin for pointing out and sharing with us a simplification in the proof of Proposition \ref{prop:AIB}. We would like to thank Ramon van Handel for many helpful discussions on the articles \cite{BandeiraVanHandel, HLY17}.
C.\ B.\ is supported by grants ANR-14-CE25-0014 and ANR-16-CE40-0024-01. A.\ K.\ gratefully acknowledges funding from the European Research Council (ERC) under the European Union's Horizon 2020 research and innovation programme (grant agreement No.\ 715539\_RandMat) and from the Swiss National Science Foundation through the SwissMAP grant.

\section{Main results}

For a positive integer $n$ we abbreviate $[n] \deq \{1,2,\dots, n\}$. For a square matrix $M$ we denote by $\sigma(M)$ the spectrum of $M$ and by $\rho(M) \deq \max_{\lambda \in \sigma(M)} \abs{\lambda}$ the spectral radius of $M$. Note that for Hermitian $M$ we have $\rho(M) = \norm{M} \deq \norm{M}_{2 \to 2}$, the operator norm of $M$. In some heuristic discussions on the limit $n \to \infty$, we use the usual little $o$ notation, and write $x \ll y$ to mean $x / y \to 0$.

\begin{Def} \label{def:B}
Let $H = (H_{ij})_{i,j \in [n]} \in M_n( \C)$ be a matrix with complex entries $H_{ij} \in \C$.
The \emph{nonbacktracking matrix associated with $H$} is the matrix $B = (B_{ef})_{e,f \in [n]^2} \in M_{n^2} ( \C)$ defined for $e = (i,j) \in [n]^2$ and $f =  (k,l) \in [n]^2$ by
  \begin{equation}\label{eq:defB}
 B_{e   f} \deq H_{kl}\ind{j  = k} \ind{i \ne l}\,.
 \end{equation}
\end{Def}

We shall need the $\ell^2 \to \ell^\infty$ and $\ell^1 \to \ell^\infty$ norms of a matrix $H \in M_n(\C)$, defined respectively as
\begin{equation} \label{def_norms_H}
\norm{H}_{2\to\infty} \deq \max_i  \sqrt{\sum_j |H_{ij}|^2}\,, \qquad \norm{H}_{1\to\infty} \deq \max_{i,j} \abs{H_{ij}}\,.
\end{equation}
The following deterministic result estimates the spectral radius of an arbitrary Hermitian matrix in terms of the spectral radius of its nonbacktracking matrix.

\beg{Th}\label{thm:From_NBmatrix_to_matrix}
For $x \geq 0$ define $f(x)$ through $f(x) = x  + x^{-1}$ if $x \geq 1$ and $f(x) = 2$ if $0 \leq x \leq 1$. Let $H \in M_n(\C)$ be a Hermitian matrix with associated nonbacktracking matrix $B$. Then
$$
\norm{H} \leq \norm{H}_{2\to\infty} \, f \pbb{\frac{\rho(B)}{ \norm{H}_{2\to\infty}}}  + 7 \norm{H}_{1\to\infty}\,.
$$
\en{Th}

Using $f(x) \leq 2 + (x-1)_+^2 $ we immediately deduce the following result.

\begin{cor}\label{cor:From_NBmatrix_to_matrix}
If $H$ is Hermitian with associated nonbacktracking matrix $B$ then $$\|H\|\le 2\norm{H}_{2\to\infty} +\f{(\rho(B)-\norm{H}_{2\to\infty} )_+^2}{\norm{H}_{2\to\infty} }+7\norm{H}_{1\to\infty}\,.$$
\en{cor}

We shall study the spectral radius  $\rho(B)$ for the following class of random matrices.
\begin{assum} \label{ass:H}
Let $H \in M_n(\C)$ be a Hermitian random matrix whose upper triangular entries $(H_{ij})_{1 \leq i \leq j \leq n}$ are independent mean-zero random variables. Moreover, suppose that there exist $q >0$ and $\kappa \geq 1$ such that
\begin{equation} \label{H_moment_ass}
\max_i\sum_j \E |H_{ij}|^2 \le 1\,, \qquad \max_{i,j} \E |H_{ij}|^2 \leq \frac{\kappa}{n}\,, \qquad \max_{i,j} \abs{H_{ij}} \leq \frac{1}{q} \text{ a.s.}
\end{equation}
\end{assum}

For the interpretation of these parameters, consider first the simple case where $A$ is the adjacency matrix of the homogeneous Erd\H{o}s-R\'enyi graph $G(n,d/n)$ with $1 \leq d \leq n$. Define $H \deq d^{-1/2} (A - \E A)$. Then it is easy to check that Assumption \ref{ass:H} holds with $q = \sqrt{d}$ and $\kappa = 1$. Thus, the parameter $q$ controls the sparsity of $A$ (the smaller it is, the sparser $A$ may be). Moreover, if the graph associated with $A$ is inhomogeneous, the variances $\E |H_{ij}|^2$ depend on $i$ and $j$, and $\kappa$ has to be chosen larger than $1$. Thus, $\kappa$ controls the structure or inhomogeneity of $A$ (the closer to $1$ it is, the less structured, or more homogeneous, $A$ is).

\begin{Th} \label{thm:rhoB}
There are universal constants $C, c > 0$ such that the following holds. Suppose that $H$ satisfies Assumption \ref{ass:H} and that $B$ is the nonbacktracking matrix associated with $H$. Then for $1 \vee q \leq n^{1/{10}} \kappa^{-1/{9}}$ and $\epsilon \geq 0$ we have 
\begin{equation*}
\bb P \pb{\rho(B) \geq 1 + \epsilon} \leq C n^{3 - c  q \log (1+\eps)}\,.
\end{equation*}
\end{Th}

\begin{rmk} \label{rem:q_bound}
Note that if Assumption \ref{ass:H} holds for some $q >0$ then it also holds for any $0 < \hat q \leq q$. Hence, if $\kappa^{1/9}  \leq n^{1/10}$, we may apply Theorem \ref{thm:rhoB} to $\hat q = q \wedge  (n^{1/{10}} \kappa^{-1/{9}})$. 
\end{rmk}

From Theorems \ref{thm:From_NBmatrix_to_matrix} and \ref{thm:rhoB} it is not hard to conclude an upper bound on the norm of $H$. We first note that general concentration inequalities imply that there exists a universal constant $c >0$ such that, if $H$ satisfies Assumption \ref{ass:H} then for all $t \geq 0$ we have
\begin{equation}\label{eq:BLM}
\bb P \pb{ \absb{ \norm{H}  - \bb E \norm{H}  }  \geq t }  \leq 2 e^{ - c q^2 t^2}\,.
\end{equation}
See for instance \cite[Examples 3.14 and 6.8]{BLM}; the concentration inequality in \cite[Example 6.8]{BLM} is given for the largest eigenvalue $\lambda_{\max} (H) = \sup_{ \| x \|_2 = 1}  \langle x , H x \rangle$, but the same argument applies to $\|H \| = \sup_{ \| x \|_2 = \| y \|_2 =  1} \langle y , H x \rangle$.

The estimate \eqref{eq:BLM} shows that, up to subgaussian fluctuations of order $1/q$, it is sufficient to control the expectation of the norm of $H$. Our main result in this direction is the following. 
 \begin{Th} \label{thm:norm_H}
There is a universal constant $C > 0$ such that the following holds. Suppose that $H$ satisfies Assumption \ref{ass:H}. Then for $1 \vee q \leq n^{1/{10}} \kappa^{-1/{9}}$, we have
 \begin{equation} \label{est_EH}
\bb E  \norm{H}  \leq 2 + C\frac{\eta}{\sqrt{1 \vee \log\eta }}\,, \qquad \hbox{with } \quad  \eta\deq\f{\sqrt{\log n}}{q}\,.
\end{equation}
In particular, 
\begin{equation} \label{est_EH_weak}
\bb E  \norm{H}  \leq 2 + C  \frac{\sqrt{\log n}}{q}\,.
\end{equation}
Moreover, under the same hypotheses, if in addition $q \geq C$ and $\max_{i} \sum_{j} \bb E |H_{ij} |^2 = 1$ then 
 \begin{equation}\label{eq:mainH}
\bb E  \norm{H}  \leq \bb E \norm{H}_{2\to\infty} \PAR{ 2 +  \frac C q }\,.
\end{equation}
\end{Th}

Theorem \ref{thm:norm_H} shows that with high probability, $\| H \| \leq  2 +o (1)$ as soon as $q^2  \gg \log n$ and $\kappa \ll n^{12/13}$ (See also Remark \ref{rem:kappa} below).
As illustrated in \cite[Theorem 2.2]{Khorunzhy2001} and \cite[Corollary 1.4]{bbkER}, 
this statement is sharp. We remark that for entries with symmetric distributions, a related bound (without sharp error bounds) was obtained in \cite[Theorem 2.1]{Khorunzhy2001}. (However, the assumption of symmetric distributions rules out random graphs.)

\begin{rmk} \label{rem:very sparse regime}
Theorem \ref{thm:norm_H} also provides an estimate for the very sparse regime $q^2 \leq \log n$. Writing $q^2 = \eta^{-2} \log n$ for $\eta \geq 1$ we find from Theorem \ref{thm:norm_H} and \eqref{eq:BLM} that
\begin{equation*}
\norm{H} \leq C \frac{\eta}{\sqrt{\log \eta}}
\end{equation*}
with high probability. On the other hand, in \cite{bbkER} we proved that for $A$ the adjacency matrix of $G(n,d/n)$ with $d = \eta^{-2} \log n$ and $H = (A - \E A)/q$ we have
\begin{equation*}
\norm{H} \sim \frac{\eta}{\sqrt{2 \log \eta}}
\end{equation*}
with high probability for $\eta \gg 1$ (where $a \sim b$ denotes that $a/b \to 1$ as $n \to \infty$). Thus, in the very sparse regime $q^2 \leq \log n$ and for general sparse random matrices, Theorem \ref{thm:norm_H} yields bounds whose dependence on $n$ and $d$ is optimal up to a universal constant for the example of the Erd\H{o}s-R\'enyi graph.
\end{rmk}

\begin{rmk}
For any Hermitian matrix $H$ we always have $\norm{H}_{2 \to \infty} \leq \norm{H}$. The estimate \eqref{eq:mainH} may be regarded as a probabilistic counterpart to this statement, saying that with high probability the converse estimate is also true up to a constant.
\end{rmk}

\begin{rmk}
Stated explicitly without referring to Assumption \ref{ass:H}, for instance the estimate \eqref{est_EH_weak} states that for any Hermitian matrix $H$ with independent centred entries, if
\begin{equation*}
\max_{i,j} \E \abs{H_{ij}}^2 \cdot \pbb{\max_i \sum_j \E \abs{H_{ij}}^2}^{7/2} \leq n^{-1/10} \normB{\max_{i,j} \abs{H_{ij}}}_{L^\infty}^{9}
\end{equation*}
 then 
\begin{equation*}
\E \norm{H} \leq 2 \sqrt{\max_i \sum_j \E \abs{H_{ij}}^2} + C \normB{\max_{i,j} \abs{H_{ij}}}_{L^\infty} \sqrt{\log n}\,.
\end{equation*}
The same remains valid if the quantity $\normb{\max_{i,j} \abs{H_{ij}}}_{L^\infty}$ is replaced by any larger number.
\end{rmk}

Our techniques also apply to non-Hermitian matrices with independent entries.

\begin{Th} \label{thm:norm_hH}
Let $H \in M_n(\C)$ be a random matrix whose entries $(H_{ij})_{1 \leq i , j \leq n}$ are independent mean-zero random variables. Moreover, suppose that there exist $q >0$ and $\kappa \geq 1$ such that \eqref{H_moment_ass} holds. Then for $1 \vee q \leq n^{1/{10}} \kappa^{-1/{9}}$ and $\epsilon\geq 0$ we have 
\begin{equation*}
\bb P \pb{\rho(H) \geq 1 + \epsilon} \leq C n^{2 - c q \log ( 1+ \eps)}\,,
\end{equation*} for some universal positive constants $C,c$.
\end{Th}

\begin{rmk}
In Theorems \ref{thm:rhoB}, \ref{thm:norm_H}, and \ref{thm:norm_hH}, the almost sure last condition of \eqref{H_moment_ass} can be easily relaxed by a truncation argument of the entries of $H$  (see for example \cite{MR950344}). We do not pursue this generalization any further. 
\end{rmk}

\begin{rmk} \label{rem:kappa}
It can easily be seen from the proofs that in Theorems \ref{thm:rhoB}, \ref{thm:norm_H}, and \ref{thm:norm_hH}, the   condition  $q\le n^{1/13}\ka^{-1/12}$ can be   relaxed to  $q\le n^{(1-\del_0)/12}\ka^{-1/12}$ for any fixed   $\del_0\in (0,1)$. Then, the constants involved in the statements depend on $\del_0$. Consequently, our main results hold for $\kappa \leq n^{1 - c}$ for any fixed $c > 0$. The applications to inhomogeneous Erd\H{o}s-R\'enyi graphs
in Section \ref{sec:ER} can be slightly strengthened in this direction; we shall not pursue this direction further.
Moreover, since the parameter $q$ only appears as an upper bound in \eqref{H_moment_ass}, our results clearly apply without an upper bound on $q$; see Remark \ref{rem:q_bound}.
\end{rmk}

\subsection{Overview of proofs}

We conclude this section by giving an overview of the proofs of our main results -- Theorems \ref{thm:From_NBmatrix_to_matrix}, \ref{thm:rhoB}, \ref{thm:norm_H}, and \ref{thm:norm_hH}. Theorem \ref{thm:From_NBmatrix_to_matrix} is proved in Section \ref{sec:proof_22}. The first main ingredient is a variant of the Ihara-Bass formula, Lemma \ref{le:IB}, which relates the spectra of $H$ and $B$ by characterizing the spectrum of $B$ as the singularity set of an explicit matrix function, denoted $M(\lambda) - H(\lambda)$. The second main ingredient is an analysis of the matrix function $M(\lambda) - H(\lambda)$, estimating its singularities in terms of eigenvalues of $H$, hence concluding the proof of Theorem \ref{thm:From_NBmatrix_to_matrix}. The proof of Theorems \ref{thm:rhoB} and \ref{thm:norm_H} are given in Section \ref{thm:norm_H}, where we estimate the spectral radius of $B$. The main work is to estimate $\E \tr B^{\ell} B^{* \ell}$ for large $\ell$, which is performed in Proposition \ref{prop:BB_bound}. It is proved using a variant of the classical moment method, exploiting the nonbacktracking property of $B$ to significantly reduce the combinatorics of the resulting graphs. Theorem \ref{thm:rhoB} then easily follows using Markov's inequality, and Theorem \ref{thm:norm_H} follows by an elementary argument using Theorems \ref{thm:From_NBmatrix_to_matrix} and \ref{thm:rhoB}. Finally, in Section \ref{sec:6} we prove Theorem \ref{thm:norm_hH}. Its proof is remarkably similar to that of Proposition \ref{prop:BB_bound}. Indeed, since the entries $H_{ij}$ and $H_{ji}$ for $i \neq j$ are independent, to leading order, the resulting walks on graphs do not backtrack. This results in walks analogous to those of Section \ref{thm:norm_H}, and we may take over much of the arguments developed there to conclude the proof of Theorem \ref{thm:norm_hH}.

  \section{Application to inhomogeneous Erd\H{o}s-R\'enyi graphs} \label{sec:ER}
  
 \subsection{Undirected graphs}
An important example of a matrix $H$ satisfying Assumption \ref{ass:H} is the (centred and rescaled) adjacency matrix of an inhomogeneous (undirected) Erd\H{o}s-R\'enyi random graph, where each edge $\{i,j\}$, $1\le i < j\le n$, is included with \pro $p_{ij}$, independently of the others. Let $A$ be its adjacency matrix and set $$H\deq d^{-1/2}(A-\E A)\,,$$
where
\begin{equation} \label{def_d}
d\deq \max_i\sum_j p_{ij}
\end{equation}
is the maximal expected degree.
Then for each $i,j$ and each $k\ge 2$, we have $|H_{ij}| \le 1 / \sqrt d$ and $\bb E |H_{ij}|^2 \leq  p_{ij } / d$. We set
\begin{equation} \label{def_q_er}
q \deq \sqrt{d} \wedge n^{1/{10}} \kappa^{-1/{9}}\,, \qquad \kappa \deq \f{ \max_{i,j}  p_{ij}}{d/n}\,.
\end{equation}

\begin{rmk} \label{rem:small_deg}
A scenario of particular interest is when the probabilities $p_{ij}$ are all comparable (i.e. $\kappa = O(1)$) and the typical expected degree $d$ is not too large, in the sense that $d \ll n^{1/{5}}$. In that case we have $q = \sqrt{d}$.
\end{rmk}

We deduce the following immediate  consequence of Theorem \ref{thm:norm_H}.

\beg{Th}
\label{th:OER}
Let $A$ be the adjacency matrix of an inhomogeneous Erd\H{o}s-R\'enyi graph, with $d$ and $q$ defined as in \eqref{def_d} and \eqref{def_q_er}. If $q \geq 1$ then
\be\la{eq:Th.2.1}
\frac{\bb E  \norm{ A - \bb E A}}{ \sqrt d}  \leq 2  + C\frac{\eta}{\sqrt{1 \vee \log\eta }}\,, \qquad \hbox{with } \quad  \eta\deq\f{\sqrt{\log n}}{q},
\ee
for some universal constant $C  >0$. If the maximal expected degree is not too large in the sense that
\begin{equation} \label{small_d}
  d \geq 1   \AND d^{7/2} \max_{i,j} p_{ij}\le n^{-1/10}\,,
\end{equation}
which in particular includes the scenario of Remark \ref{rem:small_deg}, then
\begin{equation}\label{eq:kjde}
\bb E  \norm{ A - \bb E A}  \leq 2 \sqrt d  + C \sqrt{ \frac{\log n} { 1 \vee \log \pb{\frac{\log n}{ d}} }} \,.
\end{equation}
\en{Th}

 In \cite{MR2155709,MR2853072}, it is proved that 
 $\norm{H} \le C$ if $d $ is at least of order $ \log n$. Theorem \ref{th:OER}  retrieves this result and states also that    
 $\norm{H} \leq 2 + o(1) $ as soon as $d \gg \log n$. 
 In the homogenous  case (i.e.\ when for all $i,j$, $p_{ij}=d/n$), it is well known \cite{KKPS,KnowlesBenaych} that   the empirical distribution of the eigenvalues of $H$ converges weakly to the semi-circular law with support $[-2,2]$ as soon as $d\gg 1$. Theorem \ref{th:OER} then also gives the   convergence of the extreme eigenvalues to the boundary of $[-2,2]$ as soon as $d\gg \log n$.
In the companion paper \cite{bbkER}, we study the largest eigenvalues of $H$ in the regime $d \ll \log n$, showing that a crossover occurs at $d \asymp \log n$. Interestingly, in at least the regime $d \ll \log n$,  \cite[Corollary 1.4]{bbkER} implies that the upper bound \eqref{eq:kjde} is sharp up to the multiplicative constant $C$; see Remark \ref{rem:very sparse regime}.

Theorem \ref{th:OER} can be used in statistical inference on graphs, where one wishes to infer information about $\E A$ from a single observation of $A$.  Weyl's inequality for eigenvalues implies that, for any integer $1 \leq k \leq n$, the $k$-th largest eigenvalue (counting multiplicities) of $\bb E A$ and $A$ differ by at most $\norm{ A - \bb E A}$ (see for example \cite{MR1477662}). Hence, 
Theorem \ref{th:OER} shows that the location of an eigenvalue $\lambda$ of $\bb E A $ can be effectively estimated from the spectrum of $A$ as soon as $|\lambda| / \sqrt d$ is much larger than the right-hand side of \eqre{eq:Th.2.1}.
In particular, if $|\lambda|$ is of order $d$, the condition reads \be\la{inferencecond:d}d \gg \sqrt{\f{\log n}{\log\log n}}\,.\ee
As mentioned above, \cite{MR2155709,MR2853072} require the stronger condition $d \gg \log n$. We also recall that classical tools from perturbation theory assert that also the corresponding eigenspaces of $A$ and $\bb E A$ are close when $\norm{ A - \bb E A}$ is small compared to the spectral gap around the eigenvalue of $A$ under consideration (for instance, the Davis-Kahan Theorem \cite{MR0264450} gives a precise quantitative statement of this kind). We may summarize our discussion with the following corollary,
which can be used in statistical inference on graphs when \eqre{inferencecond:d} is satisfied. For simplicity of presentation, we focus on the case where the expected maximal degree is not too large, that is when \eqref{small_d} holds.

\begin{cor}\la{CorInference1016}
Let $A$ be the adjacency matrix of an inhomogeneous Erd\H{o}s-R\'enyi graph, with $d$ defined as in \eqref{def_d}. Suppose that \eqref{small_d} holds and $n \geq 3$. Then for some universal constant $C  >0$ and for any $0 < \eps < 1$, if \be\la{eq:lbd291216}d \geq\f{C}{\eps^2} \sqrt{\f{\log n}{ \log\log n}}\ee then
$$
  \norm{ A - \bb E A}  \leq \eps d\,,
$$
with probability at least $1 - 2 \exp ( - \eps^2 d^3 / C)$. 
\end{cor}
\bpr
This is an elementary argument using $\eqre{eq:kjde}$, the condition \eqre{eq:lbd291216}, and \eqref{eq:BLM}. We omit the details.
\epr
  
   \subsection{Directed graphs}\la{sec:OG}
  An important example of a matrix $H$ satisfying the assumptions of Theorem \ref{thm:norm_hH} is the (centred and rescaled) adjacency matrix of an inhomogeneous directed Erd\H{o}s-R\'enyi random graph, where each directed edge $(i,j)$, $1\le i , j\le n$, is included with \pro $p_{ij}$, independently of the others. Let $A$ be its adjacency matrix and set $$H\deq d^{-1/2}(A-\E A)\,,$$
  with $d$ given in \eqref{def_d}.
We deduce the following immediate  consequence of Theorem  \ref{thm:norm_H}.

\beg{Th}
\label{th:ER}
Let $A$ be the adjacency matrix of an inhomogeneous directed Erd\H{o}s-R\'enyi graph, with $d$ and $q$ defined as in \eqref{def_d} and \eqref{def_q_er}. If $q \geq 1$ then
\begin{equation*}
\bb P \pb{\rho\pb{d^{-1/2}(A - \E A)} \geq 1 + \epsilon} \leq C n^{2 - c q \log ( 1+ \eps)}\,,
\end{equation*} for any $\epsilon \geq 0$ and some universal positive constants $C,c$.
If the maximal expected degree is not too large in the sense that \eqref{small_d} holds, then we have
\begin{equation*}
\bb P \pb{\rho\pb{d^{-1/2}(A - \E A)} \geq 1 + \epsilon} \leq C n^{2 - c \sqrt{d} \log ( 1+ \eps)}
\end{equation*}
for all $\epsilon \geq 0$.
\en{Th}

The following corollary can be deduced directly, using the version of the Bauer-Fike theorem given in \cite{BLMNBNRRG}.

\beg{cor} Suppose that the assumptions of Theorem \ref{th:ER} hold, and suppose moreover that  for all $i,j$ we have $p_{ij}=d/n$ where $d \leq n^{1/{5}}$. Then for any $\eps\ge 0$ \st $2(1+\eps)<\sqrt{d}$, with \pro at least $1-C n^{2 - c \sqrt{d} \log ( 1+ \eps)}$, $A$ has exactly one eigenvalue at distance at most $(1+\eps )\sqrt{d}$ from $d$ and all other eigenvalues with modulus at most $(1+\eps )\sqrt{d}$.
\en{cor}

\section{Comparison of spectra of $H$ and $B$ and proof of Theorem \ref{thm:From_NBmatrix_to_matrix}} \label{sec:proof_22}

The rest of this paper is devoted to the proofs of Theorems \ref{thm:From_NBmatrix_to_matrix}, \ref{thm:rhoB}, \ref{thm:norm_H}, and \ref{thm:norm_hH}.
  
\subsection{An Ihara-Bass-type formula}

The following lemma is a variant of the Ihara-Bass formula. It is inspired by \cite{bethehessian} and generalizes ideas from \cite{MR3174850}. It is essentially contained in Theorem 2 of \cite{WatFuk}; for the convenience of the reader and to keep this paper self-contained, we give the simple proof.

\begin{lemma}\label{le:IB}
Let $H \in M_{n} ( \C)$ with associated nonbacktracking matrix $B$ and let  $\lambda \in \C$ satisfy
$\lambda^2 \neq H_{ij} H_{ji}$ for all $i,j \in [n]$.
Define the matrices $H(\lambda)$ and $M(\lambda) = \diag( m_i(\lambda))_{i \in [n]}$ through
\begin{equation} \label{def_AM}
H_{ij}(\lambda) \;\deq\; \frac{\lambda H_{ij}}{\lambda^2 - H_{ij}H_{ji}} \,, \qquad m_i(\lambda) \;\deq\; 1  + \sum_{k} \frac{H_{ik}H_{ki} }{\lambda^2 - H_{ik}H_{ki} }\,.
\end{equation} 
Then $\lambda \in \sigma(B)$ if and only if  $\det (M(\lambda) - H(\lambda)) = 0$.
\end{lemma}

\begin{proof}
We abbreviate $ij \equiv (i,j) \in [n]^2$. Let $\lambda \in \sigma(B)$ be an eigenvalue of $B$ with eigenvector $x \in \C^{[n]^2}$, i.e.\ $B x = \lambda x$, which reads in components
\begin{equation}\label{eq:eigB}
\lambda x_{ji} = \sum_{k \ne j} H_{ik} x_{ik}
\end{equation}
for all $i,j \in [n]$.
 We define $y \in \C^n$ by, for each $i$,  
 $$
 y_i \deq \sum_k H_{ik} x_{ik}\,. 
 $$
The eigenvalue equation $\lambda x = B x$ reads
 $$
 \lambda x_{ji} = y_i - H_{ij}  x_{ij}\,. 
 $$
Exchanging $i$ and $j$, we obtain
 $$
 \lambda x_{ij} = y_j - H_{ji} x_{ji}\,,
 $$
from which we deduce
$$
 \lambda^2 x_{ji} =  \lambda y_i - H_{ij}   \lambda x_{ij}  =   \lambda y_i - H_{ij}  y_j + H_{ij}  H_{ji} x_{ji}\,.
 $$
 Hence, because $\lambda^2 \neq H_{ij} H_{ji}$,
\begin{equation} \label{x_defy}
 x_{ji} = \frac{ \lambda y_i - H_{ij} y_j }{\lambda^2 - H_{ij} H_{ji}}\,. 
\end{equation}
We see from this last expression
that $y \ne 0$ if $x \ne 0$. We plug this last expression into \eqref{eq:eigB} and get
 $$
 \frac{\lambda^2 y_i }{ \lambda^2 - H_{ij} H_{ji}}  - \frac{ \lambda H_{ij} y_j }{\lambda^2 - H_{ij} H_{ji}} =  \sum_{k \ne j} \frac{\lambda H_{ik} y_k}{ \lambda^2 - H_{ik} H_{ki} }  - \sum_{k \ne j} \frac{ H_{ik} H_{ki} y_i }{\lambda^2 - H_{ik} H_{ki}}\,,
 $$
 i.e.
  $$
 \frac{\lambda^2 y_i }{ \lambda^2 - H_{ij} H_{ji}}
 =  \sum_{k  } \frac{\lambda H_{ik} y_k}{ \lambda^2 - H_{ik} H_{ki} }  - \sum_{k \ne j} \frac{ H_{ik} H_{ki} y_i }{\lambda^2 - H_{ik} H_{ki}}\,.
 $$
 We conclude that
  $$y_i=
 \frac{\lambda^2 y_i }{ \lambda^2 - H_{ij} H_{ji}}   - \frac{ H_{ij} H_{ji} y_i }{\lambda^2 - H_{ij} H_{ji}} =  \sum_{k  } \frac{\lambda H_{ik} y_k}{ \lambda^2 - H_{ik} H_{ki} }  - \sum_{k } \frac{ H_{ik} H_{ki} y_i }{\lambda^2 - H_{ik} H_{ki}}\,.  
 $$
 Hence,
 $$
y_i \PAR{ 1 +  \sum_{k} \frac{ H_{ik} H_{ki} }{\lambda^2 - H_{ik} H_{ki}}  } -  \sum_{k} \frac{\lambda H_{ik} y_k}{ \lambda^2 - H_{ik} H_{ki} }  = 0\,,
 $$
which proves that $0$ is an eigenvalue of $M(\lambda) - H(\lambda)$.

Conversely, if $0$ is an eigenvalue of $M(\lambda) - H(\lambda)$ with eigenvector $y$, we define $x$ through \eqref{x_defy}.
Then the above computation also implies that $x$ satisfies \eqref{eq:eigB}, i.e.\ $Bx = \lambda x$, so that $\lambda \in \sigma(B)$.
\end{proof}

\subsection{Comparison of spectra and proof of Theorem \ref{thm:From_NBmatrix_to_matrix}}

We use the notation $M \preceq N$ for Hermitian matrices $M$ and $N$ to mean that $N - M$ is a positive matrix. The key estimate behind the proof of Theorem \ref{thm:From_NBmatrix_to_matrix} is the following result. We thank Sasha Sodin for his help in simplifying its proof.

\begin{proposition}\label{prop:AIB}
 Let $H \in M_n(\C)$ be a Hermitian matrix with associated nonbacktracking matrix $B$. 
Suppose that there exists $\delta \in [0, 1]$ such that 
\begin{equation}\label{eq:leAIB}
\max_{i,j} |H_{ij}| \leq \delta \qquad \text{and} \qquad \max_{i}  \sum_j |H_{ij}|^2  \leq 1 + \delta\,.  
\end{equation}
Let \begin{equation}\label{eq:deflam0}\lam_0\deq \max\hb{1+\sqrt\del,\max(\si(B)\cap\R)}\,.\ee
Then 
$$
H\preceq \lam_0+\ff{\lam_0}+6\del.
$$
\end{proposition}

\begin{proof}
Let $H(\lambda)$ and $M(\lambda)$ be the matrices defined in Lemma \ref{le:IB}. First note that   $0\preceq M(\lam_0) - H(\lam_0)$. Indeed,  as $\lambda \to+\infty$,
\begin{equation*}
M(\lambda) - H(\lambda)= I+O(\lambda^{-1})\,,
\end{equation*}
so that for $\lambda$ large enough $M(\lambda) - H(\lambda)$ is positive definite. But by Lemma \ref{le:IB}, the real zeros of $\lambda \mapsto \det(M(\lambda) - H(\lambda))$ are the real eigenvalues of $B$, which implies that  for $\lambda>\lam_0$, $\det(M(\lambda) - H(\lambda)) >0$. By continuity, we conclude that $M(\lambda) - H(\lambda)$ is positive definite for any $\lambda>\lam_0$.

Next, a direct computation  shows that for any $\lambda\ge  1 +  \sqrt{\delta}$, we have
\begin{equation*}
|\lambda H_{ij}(\lambda)-H_{ij}|=|H_{ij}|^3\ff{\lambda^{2}-|H_{ij}|^2}\le\del |H_{ij}|^2\f{1}{\lambda^2-\del^2}\le \del |H_{ij}|^2\f{1}{1+\del+2\sqrt \del-\del^2}\le \del |H_{ij}|^2\,.
\end{equation*}
We deduce (by the Schur test or the Gershgorin circle theorem) that $\|\lambda H(\lambda)-H\|\le \del\max_{i}\sum_j|H_{ij}|^2\le 2\del$. Another computation shows that for any $\lambda\ge  1 +  \sqrt{\delta}$,   we have \begin{multline*} \lambda m_{i}(\lambda)- \pbb{\lambda+\ff \lambda}= \frac{1}{\lambda}\pbb{\sum_{k} \frac{|H_{ik}|^2 }{1 - \lambda^{-2} |H_{ik}|^2}-1}\;\le\; \ff \lambda\pbb{  \frac{1+\del}{1 - \lambda^{-2} \del^2}-1}\\ =\del\f{\lambda(1+\lambda^{-2}\del)}{\lambda^2-\del^2}\;\le\;\del\f{2\lambda}{\lambda^2-\del^2}\;\le\; 4\del\,.\end{multline*}
(In the last step, we considered separately the cases $\lambda\ge 2$ and  $\lambda<2$, for which we use $\lambda^2-\del^2\ge 1$.) From both previous computations, we deduce that    for any $\lambda\ge  1 +  \sqrt{\delta}$,   $$\lambda (M(\lambda)-H(\lambda)) \preceq \lambda+\ff \lambda-H+6\del.$$

To sum up, we have $$0\preceq \lambda_0 (M(\lam_0) - H(\lam_0)) \preceq \lam_0+\ff{\lam_0}-H+6\del\,,$$ from which the claim follows.
\end{proof}

We may now conclude the proof of Theorem \ref{thm:From_NBmatrix_to_matrix}.

\bpr[Proof of Theorem \ref{thm:From_NBmatrix_to_matrix}]
Set $\del \deq \norm{H}_{1\to\infty}$.
Note that $\delta \leq 1$ because $\norm{H}_{2\to\infty} \leq  1$. By Proposition \ref{prop:AIB}, for $\lam_1\deq \max\{1+\sqrt\del,\rho(B)\}$, we have $$
\|H\|\le f(\lam_1)+6\del\,.$$
By considering the cases $\lambda_1 = \rho(B)$ and $\lambda_1 = 1 + \sqrt{\delta}$ separately, and using that $2 \leq f(1+x)\le 2+x^2$, we easily find that if $\norm{H}_{2\to\infty} \leq  1$ then
$$
\norm{H} \leq f(\rho(B))  +7 \norm{H}_{1\to\infty}\,.
$$
The claim for arbitrary $H$ now follows by homogeneity. 
\epr

\section{Estimate of $\rho(B)$ and proof of Theorems \ref{thm:rhoB} and \ref{thm:norm_H}}

The main estimate of this section is the following result.

\begin{proposition} \label{prop:BB_bound}
Suppose that $H$ satisfies Assumption \ref{ass:H} and that $B$ is the nonbacktracking matrix associated with $H$.
There exist univeral constants $c_0 , C_0 > 0$ such that the following holds.
If $\ell \geq 1$, $q \geq 1$, and $\delta \in (0,1/3)$ satisfy
\begin{equation} \label{ell_assumptions}
\ell \leq c_0 \min \hbb{\delta q \log n, \frac{n^{1/3 - \delta}}{q ^2 \kappa^{1/3}}}
\end{equation}
then
\begin{equation*}
\E \tr B^\ell B^{*\ell} \leq C_0 n^2 \ell^{4} q^2\,.
\end{equation*}
\end{proposition}

\subsection{Proof of Proposition \ref{prop:BB_bound}} \label{sec:pfpropBB}
Throughout the following we fix $\ell$ and mostly omit it from our notation.
For any $e\in [n]^2$ and $f\in [n]^2$ we have 
\begin{equation*} (B^{\ell})_{ef}=\sum_{a_1,\dots, a_{\ell - 1} \in [n]^2} B_{e a_1}B_{a_1a_2}\cd B_{a_{\ell-1} f}\,.
\end{equation*}
By Definition \ref{def:B}, we therefore find
\begin{equation*} (B^{\ell})_{ef}=\sum_\xi H_{\xi_0\xi_1}H_{\xi_1\xi_2}\cd H_{\xi_{\ell-1}\xi_{\ell}}\,,
\end{equation*}
where the sum runs over  $\xi = (\xi_{-1},\xi_0, \ld, \xi_{\ell}) \in [n]^{\ell + 2}$ satisfying $(\xi_{-1}, \xi_0)=
e$, 
$(\xi_{\ell-1},\xi_{\ell})=
f$
and $\xi_{i-1} \ne  \xi_{i+1}$ for $i=0, \ld, \ell-1$.
Hence,
\begin{align*} \Tr B^\ell{B^\ell}^*&=\sum_{e,f\in [n]^2}|(B^\ell)_{ef}|^2\\
&=
\sum_{\xi^1, \xi^2} H_{\xi^1_0\xi^1_1}H_{\xi^1_1\xi^1_2}\cd H_{\xi^1_{\ell-1}\xi^1_{\ell}}\ovl{H}_{\xi^2_0\xi^2_1}\ovl{H}_{\xi^2_1\xi^2_2}\cd \ovl{H}_{\xi^2_{\ell-1}\xi^2_{\ell}}\,,
\end{align*}
where the sum runs over $\xi^1=(\xi^1_{-1}, \ld, \xi^1_{\ell})$, $\xi^2=(\xi^2_{-1}, \ld, \xi^2_{\ell}) \in [n]^{\ell + 2}$ such that $(\xi^1_{-1}, \xi^1_0)=(\xi^2_{-1},\xi^2_0)$, 
$(\xi^1_{\ell-1},\xi^1_{\ell})=(\xi^2_{\ell-1},\xi^2_{\ell})
$
and $\xi^1_{i-1} \ne  \xi^1_{i+1}$ and $\xi^2_{i-1} \ne  \xi^2_{i+1}$ for $i=0, \ld, \ell-1$.  Note that $\xi^1_{-1}$ does not appear as an index of $H$.
Fixing all indices except $\xi^1_{-1} = \xi^2_{-1}$, we find that the sum over $\xi^1_{-1} = \xi^2_{-1}$ is bounded by $n$. The remaining sum over $\xi^1_0, \xi^2_0, \dots, \xi^1_\ell, \xi^2_\ell$ is nonnegative. This yields the estimate
\begin{align*}
\Tr B^\ell B^{\ell *} &\le n
\sum_{\xi^1, \xi^2} H_{\xi^1_0\xi^1_1}H_{\xi^1_1\xi^1_2}\cd H_{\xi^1_{\ell-1}\xi^1_{\ell}}\ovl{H}_{\xi^2_0\xi^2_1}\ovl{H}_{\xi^2_1\xi^2_2}\cd \ovl{H}_{\xi^2_{\ell-1}\xi^2_{\ell}}
\\
&= n
\sum_{\xi^1, \xi^2} H_{\xi^1_0\xi^1_1}H_{\xi^1_1\xi^1_2}\cd H_{\xi^1_{\ell-1}\xi^1_{\ell}}   {H}_{\xi^2_{\ell}\xi^2_{\ell-1}} {H}_{\xi^2_{\ell-1}\xi^2_{\ell-2}}\cd {H}_{\xi^2_1\xi^2_0}\,,
\end{align*}
where the sum runs over $\xi^1=(\xi^1_0, \ld, \xi^1_{\ell}), \xi^2=(\xi^2_0, \ld, \xi^2_{\ell}) \in [n]^{\ell + 1}$ such that  $( \xi^1_0,\xi^1_{\ell-1},\xi^1_{\ell})=( \xi^2_0,\xi^2_{\ell-1},\xi^2_{\ell})
$
and $\xi^1_{i-1} \ne  \xi^1_{i+1}$ and $\xi^2_{i-1} \ne  \xi^2_{i+1}$ for $i=1, \ld, \ell-1$. In the second step we used that $H$ is Hermitian. Renaming the summation variables, we have
\begin{equation*}
\Tr B^\ell B^{\ell *} \leq n
\sum_{\xi \in \wt{\cal C}} H_{\xi_0\xi_1}H_{\xi_1\xi_2}\cd {H}_{\xi_{2\ell-1}\xi_{2\ell}}\,,
\end{equation*}
where
\begin{equation*}
\wt {\cal C} \deq \hb{\xi=(\xi_0, \ld, \xi_{2\ell})\in [n]^{2\ell+1} \col \xi_0=\xi_{2\ell}, \xi_{\ell-1}= \xi_{\ell+1},  \xi_{i-1}\ne \xi_{i+1} \text{ for } i\in[2\ell-1]\bck\{\ell\}}\,.
\end{equation*}
Because the entries of $H$ are independent and have mean zero, we find
\begin{equation} \label{trBB_paths}
\E \Tr B^\ell B^{\ell *} \leq n
\E \sum_{\xi \in \cal C} H_{\xi_0\xi_1}H_{\xi_1\xi_2}\cd {H}_{\xi_{2\ell-1}\xi_{2\ell}}\,,
\end{equation}
where we defined $\cal C$ as the set of $\xi \in \wt {\cal C}$ such that $\absb{\hb{i \in [2\ell] \col \{\xi_{i - 1}, \xi_i\} = \{a,b\}}} \neq 1$ for all $a,b \in [n]$. In words, for $\xi \in \cal C$ every unordered edge cannot be crossed only once by $\xi$.

In the following we shall need several basic graph-theoretic notions. Since they involve paths on multigraphs, it is important to introduce them with some care. By definition, a (vertex-labelled undirected) \emph{multigraph} $G = (V(G), E(G), \phi)$ consists of two finite sets, the set of \emph{vertices} $V(G)$ and the set of \emph{edges} $E(G)$, and a map $\phi$ from $E(G)$ to the unordered sets of one or two elements of $V(G)$. For $e \in E(G)$, the set $\phi(e)$ is the set of vertices \emph{incident} to $e$. The edge $e \in E(G)$ is a \emph{loop} if $\abs{\phi(e)} = 1$. The \emph{degree} $\deg(v)$ of a vertex $v \in V(G)$ is the number of edges to which it is incident, whereby a loop incident to $v$ counts twice. The \emph{genus} of $G$ is $g(G) \deq \abs{E(G)} - \abs{V(G)} + 1$. A \emph{path of length $l \geq 1$ in $G$} is a word $w = w_0 w_{01} w_1 w_{12} \dots w_{l-1 l} w_l$ such that $w_0, w_1, \dots, w_l \in V(G)$, $w_{01}, w_{12}, \dots, w_{l-1 l} \in E(G)$, and $\phi(w_{i - 1 i}) = \{w_{i - 1}, w_i\}$ for $i = 1, \dots, l$. We denote the length $l$ of $w$ by $\abs{w} \deq l$.
The path $w$ is \emph{closed} if $w_0 = w_l$. For $e \in E(G)$ we define the \emph{number of crossings of $e$ by $w$} to be $m_e(w) \deq \sum_{i = 1}^l \ind{w_{i-1 i} = e}$. In particular, we have $\sum_{e \in E(G)} m_e(w) = \abs{w}$.

Moreover, the multigraph $G$ is called simply a \emph{graph} if $\phi$ is injective, i.e.\ there are no multiple edges. (Note that in our convention a graph may have loops.) For a graph $G$ we may and shall identify $E(G)$ with a set of unordered pairs of $V(G)$, simply identifying $e$ and $\phi(e)$. Similarly,  we identify a path $w$ with the reduced word $w_0 w_1 \dots w_l$ only containing the vertices, since we must have $w_{i - 1 i} = \{w_{i - 1}, w_i\}$ for $i = 1, \dots, l$.

\begin{Def} \label{def_Gxi}
For $\xi \in \cal C$ we define the graph $G_{\xi}$ as
\begin{equation*}
V(G_\xi) \deq \h{\xi_i \col i = 0, \dots, 2 \ell}\,, \qquad E(G_\xi) \deq \hb{\{\xi_{i - 1}, \xi_i\} \col i = 1, \dots, 2 \ell}\,.
\end{equation*}
Thus, $\xi \equiv \xi_0 \xi_1 \cdots \xi_{2 \ell}$ is a closed path in the graph $G_\xi$.
\end{Def}
Next, we introduce an equivalence relation on $\cal C$ by saying that two paths $\xi, \tilde \xi \in \cal C$ are equivalent, denoted $\xi \sim \tilde \xi$, if and only if there exists a permutation $\tau$ of $[n]$ such that $\tau(\xi_i) = \tilde \xi_i$ for all $i = 0, \dots, 2 \ell$.
Clearly, the numbers $\abs{E(G_\xi)}$ and $\abs{V(G_\xi)}$, and hence also $g(G_\xi)$, only depend on the equivalence class of $\xi$. We denote by $[\xi] \subset \cal C$ the equivalence class of a path $\xi \in \cal C$ in the set $\cal C$.

\beg{lem} \label{lem:sum_gamma_Gamma}
For any $\bar \xi \in \cal C$ we have
$$\E\sum_{\xi \in [\bar \xi]}  H_{\xi_0\xi_1}H_{\xi_1\xi_2}\cd    {H}_{\xi_{2\ell-1}\xi_{2\ell}} \le n^{1-g(G_{\bar \xi})} \kappa^{g(G_{\bar \xi})} q^{2\abs{E(G_{\bar \xi})} - 2\ell}\,.$$
\en{lem} 

\bpr
Abbreviate $g \deq g(G_{\bar \xi})$, $a \deq \abs{E(G_{\bar \xi})}$, and $s \deq \abs{V(G_{\bar \xi})}$, so that $g=a-s+1$. Since the claim only depends on $\bar \xi$ through its equivalence class $[\bar \xi]$, we may replace $\bar \xi$ with any equivalent path of $[\bar \xi]$ obtained by relabelling the vertices. Thus, we may suppose that $V(G_{\bar \xi}) = [s]$ and that there is a spanning tree $T$ of $G_{\bar \xi}$ such that for all $t \in 2, \dots, s$ the subgraph $T \vert_{[t]}$ is a spanning tree of $[t]$. (This amounts to the requirement that the vertices $[s]$ are first explored in increasing order by the path $\bar \xi$.) Moreover, we enumerate the edges of $G_{\bar \xi}$ as $e_1, \dots, e_a$, where for $t = 1, \dots, s-1$ the edge $e_t$ is the unique edge of $T\vert_{[t+1]}$ that is not an edge of $T \vert_{[t]}$, and the edges $e_s, \dots, e_a$ are the remaining $g$ edges in some arbitrary order.

For $t=1, \ld, a$ we abbreviate $m_t \deq m_{e_t}(\bar \xi)$ for the number of crossings of $e_t$ by $\bar \xi$.
Moreover, we denote by $\cal I_{s,n}$ the set of injective maps from $[s]$ to $[n]$. Then we have
\begin{align*} \E\sum_{\xi\in [\bar \xi]}  H_{\xi_0\xi_1}H_{\xi_1\xi_2}\cd    {H}_{\xi_{2\ell-1}\xi_{2\ell}}& =\sum_{\tau \in \cal I_{s,n}} \E H_{\tau(\bar \xi_0)\tau(\bar \xi_1)}H_{\tau(\bar \xi_1)\tau(\bar \xi_2)}\cd    {H}_{\tau(\bar \xi_{2\ell-1})\tau(\bar \xi_{2\ell})} \\
   &\le  \sum_{\tau \in \cal I_{s,n}} \prod_{t=1}^a\E|H_{\tau(e_t)}^{m_t}|, 
   \end{align*}
   where we used the independence of the entries and the convention $\tau(\{x,y\}) \deq \{\tau(x),\tau(y)\}$. 
We use the estimate 
   \begin{equation}\label{Hyp0MomentsH}\max_i\sum_j\E |H_{ij}|^{k} \le \ff{q^{k-2}} \qquad (k = 2,3,\dots) \end{equation}
   for the edges $e_1, \ld, e_{s-1}$ and the estimate
   \begin{equation}\label{HypMomentsH} \max_{i,j} \E |H_{ij}|^{k} \leq \f{\kappa}{nq^{k-2}} \qquad (k = 2,3,\dots)
\end{equation}
for the edges $e_s, \ld, e_{a}$; both of these estimates follow immediately from \eqref{H_moment_ass}. Thus we find
  \begin{align*}
  \E\sum_{\xi\in [\bar \xi]}  H_{\xi_0\xi_1}H_{\xi_1\xi_2}\cd    {H}_{\xi_{2\ell-1}\xi_{2\ell}}
  &\le \prod_{t=s}^{a} \f{\kappa}{nq^{m_t-2}}
   \sum_{\tau \in \cal I_{s,n}} 
   \prod_{t=1}^{s-1}\E|H_{\tau(e_t)}^{m_t}|
   \\
  &\le  \prod_{t=s}^{a} \f{\kappa}{nq^{m_t-2}}
  \sum_{\tau \in \cal I_{s-1,n}} 
  \prod_{t=1}^{s-2}\E|H_{\tau(e_t)}^{m_t}|\ti \ff{q^{m_{s-1}-2}}
   \\
    &\le  \prod_{t=s}^{a} \f{\kappa}{nq^{m_t-2}}
    \sum_{\tau \in \cal I_{s-2,n}} 
    \prod_{t=1}^{s-3}\E|H_{\tau(e_t)}^{m_t}|\ti  \ff{q^{m_{s-2}-2}}\ff{q^{m_{s-1}-2}} \\
    &\le \cd
    \\
    &\le  \prod_{t=s}^{a} \f{\kappa}{nq^{m_t-2}} \sum_{\tau \in \cal I_{1,n}}\ff{q^{m_{1}-2}}\cd \ff{q^{m_{s-2}-2}}\ff{q^{m_{s-1}-2}} \\
    &= nq^{2a-\sum_{t=1}^am_t} (\kappa/n)^{a-s+1}\,.
   \end{align*}
We conclude the proof noting that $g=a-s+1$ and $\sum_{t=1}^am_t=2\ell$.
\epr

\begin{Def} \label{def:normal}
We say that a pair $(G,w)$ formed by a multigraph $G$ and a path $w = w_0 w_{01} w_1 w_{12} \dots w_{l-1 l} w_l$ is \emph{normal} if
\begin{enumerate}
\item
$V(G) = \{w_0, \dots, w_l\} = [s]$ where $s \deq \abs{V(G)}$;
\item
the vertices of $V(G)$ are visited in increasing order by $w$, i.e.\ if $w_i \notin \{w_0, \dots, w_{i - 1}\}$ then $w_i > w_1, \dots, w_{i - 1}$.
\end{enumerate}
\end{Def} 
Clearly, each equivalence class of $\sim$ in $\cal C$ has a unique representative $\xi$ such that $(G_\xi,\xi)$ is normal. We denote $\cal C_0 \deq \h{\xi \in \cal C \col (G_\xi,\xi) \text{ normal}}$. Thus, from \eqref{trBB_paths} and Lemma \ref{lem:sum_gamma_Gamma} we deduce that
\begin{equation} \label{BB_est1}
\E\Tr B^\ell{B^{\ell *}} \le n^2 \sum_{\xi \in \cal C_0} n^{-g(G_\xi)}\kappa^{g(G_\xi)} q^{2\abs{E(G_\xi)} - 2\ell}\,.
\end{equation}

Next, we introduce a parametrization of $\cal C_0$ obtained by deleting vertices of the graph $G_\xi$ that have degree two. In this process the two exceptional vertices $\{\xi_0,\xi_\ell\}$ are not collapsed. This process will result in a multigraph, denoted by $U(\xi)$ below. A similar construction appears in \cite{MR1137767}. We refer to Figure \ref{fig:G_U} for an illustration of the following construction. 

\begin{figure}[!ht]
\begin{center}
{\scriptsize 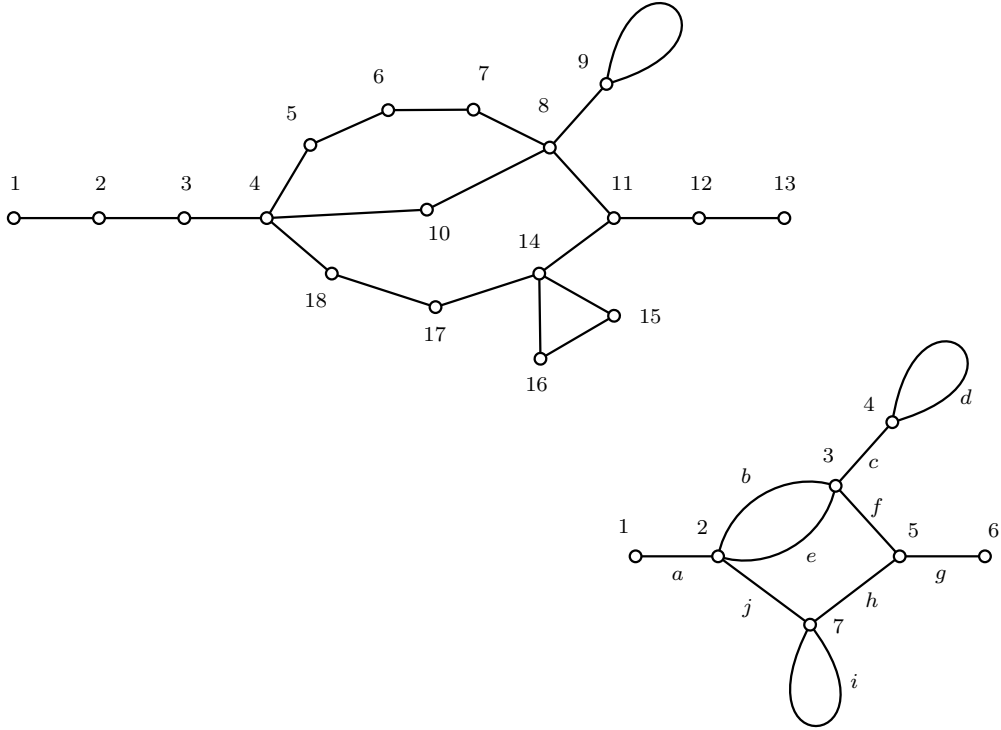}
\end{center}
\caption{
In the top diagram we draw the graph $G_\xi$ associated with the path
$\xi =$ 1,2,3,4,5,6,7,8,9,9,8,10,4,5,6,7,8,9,9,8,11,12,13,12,11,14,15,16,14,17,18,4,10,8,11,14,16,15,14,17,18,4,3,2,1. Here $\ell = 22$. Note that $G_\xi$ has a loop but no multiple edges. The number of crossings of each edge by $\xi$ is $2$, except for $\{8,9\}$ for which it is $4$. Note also that $(G_\xi,\xi)$ is normal.
In the bottom diagram we draw the multigraph $U$ associated with $\xi$. We label the vertices of $U$ by $1, \dots, 7$ and the edges of $U$ by $a, \dots, j$. The path in $U$ associated with $\xi$ is $\zeta = 1a2b3c4d4c3e2b3c4d4c3f5g6g5h7i7j2e3f5h7i7j2a1$. The number of crossings of each edge by $\zeta$ is $2$, except for $c$ for which it is $4$. Note that $\zeta$ is normal in $U$. Here $\gamma = 6$ and the weights $k$ of the edges of $U$ are given by $k_a = 3$, $k_b = 4$, $k_c = 1$, $k_d = 1$, $k_f = 1$, $k_g = 2$, $k_h = 1$, $k_i = 3$, $k_j = 3$.
\label{fig:G_U}}
\end{figure}

\begin{Def}
Let $G$ be a graph and $\cal V \subset V(G)$. Define the set
\begin{equation*}
\cal I_{\cal V}(G) \deq \hb{v \in V(G) \setminus \cal V \col \deg(v) = 2}\,.
\end{equation*}
Define the set $\Sigma_{\cal V}(G)$ to be the set of paths $w = w_0 \cdots w_l$ in $G$ such that $w_1, \dots, w_{l - 1}$ are pairwise distinct and belong to $\cal I_{\cal V}(G)$ and $w_0,w_l \notin \cal I_{\cal V}(G)$. We introduce an equivalence relation on $\Sigma_{\cal V}(G)$ by saying that $w_0 \cdots w_l$ and $w_l \cdots w_0$ are equivalent, and denote by $\Sigma'_{\cal V}(G) \deq \{[w] \col w \in \Sigma_{\cal V}(G) \}$ the set of equivalence classes.
\end{Def}

The next definition constructs a multigraph $\hat G_\xi$ from $G_\xi$, obtained by replacing every $[w] \in \Sigma'_{\{ \xi_0, \xi_\ell\} }(G)$ with an edge of $E(\hat G_\xi)$. 

\begin{Def} \label{def:U}
Let $\xi \in \cal C_0$.
Set $V(\hat G_\xi) \deq V(G_\xi) \setminus \cal I_{\{\xi_0, \xi_\ell\}}(G_\xi)$ and $E(\hat G_\xi) \deq \Sigma'_{\{\xi_0, \xi_\ell\}}(G)$ with $\phi([w]) \deq \{w_0, w_l\}$ for each $[w] \in \Sigma'_{\{\xi_0, \xi_\ell\}}(G)$.
We also assign to each edge $[w] \in E(\hat G_\xi)$ the  \emph{weight}  $\hat k_w$ to be the length of the path $w$. \end{Def}

Let now $\xi \in \cal C_0$, which is a closed path $\xi_0 \xi_1 \cdots \xi_{2 \ell}$ of length $2\ell$ in the graph $G_\xi$. Because of the nonbacktracking condition in the definition of $\cal C $ at all vertices of $G_\xi$ except $\xi_0$ and $\xi_\ell$, we find that every pair $\xi_{i-1} \xi_i$ must be contained in a word $w \in \Sigma_{\{\xi_0, \xi_\ell\}}(G_\xi)$. By writing $\xi$ as a concatenation of words from $\Sigma_{\{\xi_0, \xi_\ell\}}(G_\xi)$, we therefore conclude that the closed path $\xi = \xi_0 \xi_1 \cdots \xi_{2 \ell}$ in the graph $G_\xi$ gives rise to a closed path $\hat \xi = \hat \xi_0 \hat \xi_{01} \hat \xi_1 \hat \xi_{12} \cdots \hat \xi_{r - 1 r } \hat \xi_r$ on the multigraph $\hat G_\xi$. We stress the fundamental role of the nonbacktracking condition in the definition of $\cal C$ in the construction of $\hat \xi$; without it such a construction fails.

Summarizing, for any given $\xi \in \cal C_0$ we have constructed a triple $(\hat G_\xi, \hat \xi, \hat k)$, where $\hat G_\xi$ is a multigraph, $\hat \xi$ is a closed path in $\hat G_\xi$, and $\hat k = (\hat k_e)_{e \in E(\hat G_\xi)}$ is the family of weights of the edges of $\hat G_\xi$.

Note that $\hat \xi$ and $\hat G_\xi$ are in general not normal in the sense of Defintion \ref{def:normal}. We remedy this by setting $\tau$ to be the unique increasing bijection from $V(\hat G_\xi)$ to $\{1, \dots, \abs{V(\hat G_\xi)}\}$. Denote by $(U,\zeta,k) \equiv (U(\xi), \zeta(\xi), k(\xi))$ the triple obtained from the triple $(\hat G_\xi, \hat \xi, \hat k)$ by relabelling the vertices using $\tau$.
By definition, $\xi_0 = \tau(\xi_0) = 1$. Moreover, we set $\gamma \equiv \gamma(\xi) \deq \tau(\xi_\ell)$.

We refer to Figure \ref{fig:G_U} for an illustration of the construction of $(U,\zeta,k)$. We now collect several basic properties of the mapping $\xi \mapsto (U,\zeta,k)$. 

\begin{lemma}
The mapping $\xi \mapsto (U,\zeta,k)$ satisfies the following properties.
\begin{enumerate}
\item
The mapping $\xi \mapsto (U,\zeta,k)$ is injective on $\cal C_0$.
\item
$g(U) = g(G_\xi)$.
\item
$\zeta$ is a closed path in the multigraph $U$. It is normal in $U$ in the sense of Definition \ref{def:normal}. In particular, $\zeta_0 = \zeta_{\abs{\zeta}} = 1$.
\item
Every vertex of $V(U) \setminus \{1,\gamma\}$ has degree at least three. The vertices $1$ and $\gamma$ have degree at least one.
\item
$\abs{E(G_\xi)} = \sum_{e \in E(U)} k_e$.
\item
$m_e(\zeta) \geq 2$ for all $e \in E(U)$ and $2 \ell = \sum_{e \in E(U)} m_e(\zeta) k_e$.
\end{enumerate}
\end{lemma}

\begin{proof}
All of these properties follow immediately from the construction of $(U,\zeta,k)$. For (i), we emphasize that the requirement $\xi \in \cal C_0$ is crucial, for a relabelling of the vertices of $G_\xi$ will not change the resulting triple $(U,\zeta,k)$.
\end{proof}
\label{U_prop}

Having constructed the triple $(U,\zeta,k)$, we may now use it to estimate the right-hand side of \eqref{BB_est1}. By property (i) above, it suffices to sum over $(U,\zeta,k)$ instead of $\xi$. We find

\begin{align*} \E  \Tr B^\ell B^{\ell *}
  &\le n^2
 \sum_{(U,\zeta,k)} n^{-g(U)}q^{-\sum_{e\in E(U)}m_e(\zeta)k_e +2\sum_{e\in E(U)}k_e} \kappa^{g(U)}\\
 &=n^2
 \sum_{(U,\zeta,k)} n^{-g(U)}q^{-\sum_{  e\in E(U)}k_e(m_e(\zeta)-2)} \kappa^{g(U)}\,,
\end{align*}
where the sum ranges over all triples $(U,\zeta,k)$ obtained from all $\xi \in \cal C_0$. Since $k_e \geq 1$ and $m_e(\zeta) \geq 2$ for all $e \in E(U)$, we find
\begin{equation*}
\sum_{e\in E(U)} k_e (m_e(\zeta) -2) \geq \sum_{  e\in E(U)}(m_e(\zeta)-2) = \abs{\zeta} - 2 \abs{E(U)}\,.
\end{equation*}
Since $q \geq 1$, we therefore get
\begin{equation*}
 \E  \Tr B^\ell B^{\ell *}
  \le n^2 
 \sum_{(U,\zeta,k)} n^{-g(U)}q^{2\abs{E(U)}-\abs{\zeta}}\kappa^{g(U)}\,.
\end{equation*}
Note that the summand does not depend on $k$. For fixed $(U,\zeta)$, we may therefore estimate the sum over $k$ by estimating from above the number of families $k = (k_e)_{e \in E(U)}$ such that $k_e \geq 1$ for all $e \in E(U)$ and $\sum_{e \in E(U)} k_e m_e(\zeta) = 2 \ell$. Since $m_e(\zeta) \geq 2 \geq 1$, this is certainly bounded by the  number of families $k = (k_e)_{e \in E(U)}$ such that $k_e \geq 1$ for all $e \in E(U)$ and $\sum_{e \in E(U)} k_e = 2 \ell$, which is equal to
\begin{equation*}
\binom{2 \ell - 1}{\abs{E(U)} - 1} \leq \pbb{\frac{6 \ell}{\abs{E(U)}}}^{\abs{E(U)}}\,.
\end{equation*}
We conclude that
\begin{equation*} \E  \Tr B^\ell B^{\ell *}
  \le  n^2 
 \sum_{(U,\zeta)} \pbb{\frac{6 \ell}{\abs{E(U)}}}^{\abs{E(U)}} n^{-g(U)}q^{2\abs{E(U)}-\abs{\zeta}}\kappa^{g(U)}\,,
\end{equation*}
where the sum ranges over pairs $(U,\zeta)$ obtained from all $\xi \in \cal C_0$. We estimate this sum using the following bounds on $\abs{E(U)}$.

\begin{lemma} \label{lem:EU_bounds}
For $U$ as above we have \begin{equation*}\label{estimate1g}
%g(U) \vee 1  \le 
 \abs{E(U)}   \le 3g(U)+1 \,,
\end{equation*} 
and \begin{equation*}
 \abs{V(U)}   \le 2g(U)+2\,.
\end{equation*}

\end{lemma}
\begin{proof}
%The lower bound $g(U) \vee 1  \le  \abs{E(U)}$ follows immediately from the definition of $g(U)$.
 In order to prove the upper bound on $\abs{E(U)}$, we write
$$2\abs{E(U)}=\sum_{v\in V(U)}\op{deg}(v)\ge 2+3(|V(U)|-2)\,,$$ where we used the fact that each vertex in $V(U) \setminus \{1, \gamma\}$ has degree at least $3$, and $\{1,\gamma\}$ have degree at least one. This gives \begin{equation*}\label{estimate1s}
|V(U)|\le \f{2}{3}\abs{E(U)}+\frac{4}{3}\,,\end{equation*} which implies that $g(U) \ge \f{\abs{E(U)} - 1}{3}$ and $g(U) \ge \f{\abs{V(U)} }{2} -1$.
\end{proof}

Using Lemma \ref{lem:EU_bounds} we conclude
\begin{equation*}
\E  \Tr B^\ell B^{\ell *}
  \le n^2 \ell  q^2
 \sum_{(U,\zeta)}\pbb{\f{12 \ell}{g(U)+1}}^{3g(U)} n^{-g(U)}q^{6g(U)-\abs{\zeta}} \kappa^{g(U)} \,,
\end{equation*}
where we used that $g(U) \leq \ell$.

Next, we claim that the number of pairs $(U,\zeta)$ such that $|E(U)| \leq  e$,  $|V(U)| \leq v$ and $\zeta$ has lenght $m$ is at most
$$
e^m v^e.
$$
Indeed, we may build the path $\zeta$ directly as follows: at each of the $m$ steps of the path $\zeta$, we choose an edge numbered from one to $e$ and, for each newly visited edge, we attribute an end vertex among $v$ possible choices. Since $(U,\zeta)$ is normal and $\zeta$ visits all edges of $U$, this characterizes uniquely the pair $(U,\zeta)$. Using Lemma \ref{lem:EU_bounds}, we therefore find that the number of pairs $(U,\zeta)$ such that $U$ has genus $g$ and $\zeta$ has length $m$ is bounded by
\begin{equation*}
(3g + 1)^m (2 g + 2)^{3g+1}.
\end{equation*}
Putting everything together, we find 
\begin{align} \E  \Tr B^\ell B^{\ell *}
&\le n^2 \ell q^2
 \sum_{g=0}^\ell\sum_{m=1}^{2\ell}(3g + 1)^m (2 g + 2)^{3g+1}\pbb{\f{24 \ell}{2g+2}}^{3g} n^{-g}q^{6g-m}\kappa^g
 \notag
 \\ \label{BB_est_sum}
 &\le C n^2 \ell q^2 \sum_{m=1}^{2\ell} \pbb{\f{1}{q}}^m+ C n^2 \ell^2 q^2
 \sum_{g=1}^\ell \pbb{\frac{C \kappa \ell^3 q^6}{n}}^g \sum_{m=1}^{2\ell} \pbb{\frac{4 g}{q}}^m\,,
\end{align}
where we used that $g \leq \ell$. Since $q \geq 1$, the first term of \eqref{BB_est_sum} is bounded by $C n^2 \ell^2 q^2$.
For any $x > 0$ we have $\sum_{m = 1}^{2\ell} x^m \leq 2 \ell (1 + x^\ell)$. Thus, the second term of \eqref{BB_est_sum} is estimated by
\begin{equation} \label{BB_est_sum_1}
C n^2 \ell^{3} q^2 \sum_{g = 1}^\ell \pbb{\frac{C \kappa \ell^3 q^6}{n}}^g
+ C n^2 \ell^{3} q^2 \sum_{g = 1}^\ell \pbb{\frac{C \kappa \ell^3 q^6}{n}}^g \pbb{\frac{4 g}{q}}^{2 \ell}\,.
\end{equation}
Now, assume that \begin{equation} \label{ell_assumptions_proof}
n \geq 2 C \kappa \ell^3 q^6 \,, \qquad \ell \leq \frac{q}{8} \log \frac{n}{C \kappa \ell^3 q^6}
\end{equation}
(we will then check that  \eqref{ell_assumptions} implies \eqref{ell_assumptions_proof}).  By the assumption \eqref{ell_assumptions_proof}, the first term of \eqref{BB_est_sum_1} is estimated by $C n^2 \ell^{3} q^2$. The second term of \eqref{BB_est_sum_1} may be written as
\begin{equation*}
C n^2 \ell^{3} q^2 \sum_{g = 1}^\ell \exp \qbb{- g \log \frac{n}{C \kappa \ell^3 q^6} + 2 \ell \log \frac{4 g}{q}}\,.
\end{equation*}
The argument of the exponential is maximized for
\begin{equation*}
g = 2 \ell \bigg/ \log \frac{n}{C \kappa \ell^3 q^6}\,.
\end{equation*}
Plugging this back in, we find that if
\eqref{ell_assumptions_proof}
 holds then this maximum is reached for $g \leq 4q$ and we deduce that $\E \tr B^\ell B^{*\ell} \leq C n^2 \ell^{4} q^2$.

What remains, therefore, is to show that \eqref{ell_assumptions} with large enough $C_0$ and small enough $c_0$ implies \eqref{ell_assumptions_proof}. From \eqref{ell_assumptions} we find
\begin{equation} \label{n_geq_est}
n \geq \pbb{\frac{\kappa \ell^3 q^6}{c_0^3}}^{\frac{1}{1 - 3 \delta}}\,,
\end{equation}
which implies the first estimate of \eqref{ell_assumptions_proof} for small enough $c_0$. Moreover, \eqref{n_geq_est} yields
\begin{equation*}
n^{3 \delta} \leq  \frac{c_0^3  n}{\kappa \ell^3 q^6}\,.
\end{equation*}
We conclude from \eqref{ell_assumptions} that
\begin{equation*}
\ell \leq c_0 \delta q \log n = \frac{c_0 q}{3} \log n^{3 \delta} \leq \frac{c_0 q}{3} \log \pbb{\frac{c_0^3 n}{\kappa \ell^3 q^6}}\,,
\end{equation*}
which implies the second estimate of \eqref{ell_assumptions_proof} for small enough $c_0$. This concludes the proof of Proposition \ref{prop:BB_bound}.

\subsection{Proof of Theorem \ref{thm:rhoB}} \label{sec:pfTh14}
This is a simple application of Proposition \ref{prop:BB_bound} and Markov's inequality. We have
\begin{equation}\label{eq:fromTrace_to_SpecRadius}
\rho(B)=\rho(B^{\ell})^{1/\ell}\le \|B^{\ell}\|^{1/\ell}=\|B^{\ell}B^{*\ell}\|^{1/(2\ell)}\le( \Tr (B^{\ell}{B^*}^{\ell}))^{1/(2\ell)}\,.
\end{equation}
Thus we get
\begin{align*}
\bb P(\rho(B) > 1 + \epsilon) \leq \bb P \pb{\tr B^\ell B^{* \ell} > (1+ \epsilon)^{2 \ell}} \leq \frac{\E \tr B^\ell B^{* \ell}}{(1 + \epsilon)^{2 \ell}}\,.
\end{align*}
Choosing $\delta \deq \frac{1}{30}$ in Proposition \ref{prop:BB_bound}, we find by assumption of Theorem \ref{thm:rhoB} that \eqref{ell_assumptions} holds with $\ell \deq  \lceil \frac c 2 q \log n \rceil$, with some universal constant $c > 0$. We therefore find from Proposition \ref{prop:BB_bound}
\begin{equation*}
\bb P(\rho(B) > 1 + \epsilon) \leq C n^2 q^{6}(\log n)^4 ( 1+ \eps)^{- c  q \log n} \leq C n^{3 - c  q \log ( 1+ \eps) }\,,
\end{equation*}
as claimed. This concludes the proof of Theorem \ref{thm:rhoB}.

\subsection{Proof of Theorem \ref{thm:norm_H}}

We start by proving the following intermediate result.
\begin{proposition} \label{prop:norm_H0}
There are universal constants $C, c > 0$ such that the following holds. Suppose that $H$ satisfies Assumption \ref{ass:H}. Then for $1 \leq q \leq n^{1/{10}} \kappa^{-1/{9}}$ and $\delta \geq 0$, we have
\begin{equation}\label{eq:propo4.9}
\bb P \pbb{\norm{H} \ge 2 +  C \delta + \frac{C}{q}} \leq C n^{3 - c q \log ( 1 + \sqrt{\delta})} + n\mathrm e^{- q^2h(\del \vee \delta^2)}\,,
\end{equation} where $h(\del) \deq (1+\del)\log(1+\del)-\del$.
\end{proposition}

\bpr
The proof is a combination of Corollary \ref{cor:From_NBmatrix_to_matrix} and Theorem \ref{thm:rhoB}. By assumption, we have $\norm{H}_{1\to\infty} \leq 1/q$. We estimate $\norm{H}_{2\to\infty}$ using Bennett's inequality \cite[Theorem 2.9]{BLM}. Fix $i \in [n]$ and define the independent random variables $X_1, \dots, X_n$, where $X_j \deq \abs{H_{ij}}^2$. By assumption on $H$ we have $X_j \leq 1/q^2$ and $\sum_j \E X_j^2 \leq 1/q^2$. Since $\sum_j \E X_j \leq 1$, we conclude from Bennett's  inequality  that
\begin{equation}\label{eq:hoeffding}
\bb P \pbb{\sum_j \abs{H_{ij}}^2 \ge 1 + t} \leq \mathrm e^{- q^2h(t) }\,.
\end{equation} 
By a union bound, we therefore deduce that $\bb P ( \norm{H}_{2\to\infty} \ge 1 + t ) \leq \bb P(\norm{H}_{2\to\infty} \ge \sqrt{1 + t \vee t^2}) \leq n \mathrm e^{- q^2h(t \vee t^2)}$. We apply this last statement to $ t = \delta$. The claim now follows from Corollary \ref{cor:From_NBmatrix_to_matrix}  and Theorem \ref{thm:rhoB} applied to $\eps^2 = \delta / C$. 
\epr

We are now ready to prove Theorem \ref{thm:norm_H}. We begin by noting that without loss of generality we can assume $q \geq 1$. Indeed, if Theorem \ref{thm:norm_H} is already established for $q \geq 1$ then in the case $q < 1$ we can apply it to $H' = H / ( 1/ q)$, with corresponding parameters $q' = 1$ and $\kappa' = \kappa q^2  \leq \kappa$, and the claim follows easily.
%Then we have, for any $\eps >0$,
%$$
%\bb P \pbb{\rho ( B) \geq \frac{2}{q} ( 1+ \eps )} \leq C n ^{3 - c \log ( 1+ \eps)}\,. 
%$$  
%By adjusting the constants $c,C$, we find that for $0 \leq q \leq 2$ we have
%$$
%\bb P \PAR{\rho ( B) \geq   1+ \eps  } \leq C n ^{3 - c q \log ( 1+ \eps)}\,. 
%$$  
%(Indeed since $q \leq 2$, this last statement is empty unless $\eps$ is at least of order $\me^{C/q}$).

From now on we therefore assume that $q \geq 1$. Let us first prove \eqref{est_EH}. We claim that for large enough $K$ and any $n,q, \ka$ satisfying \begin{equation}\label{eq:179161}1 \leq q \leq n^{1/{10}} \kappa^{-1/{9}}\AND \ka\ge 1\,,\end{equation} 
we have
\begin{equation} \label{prob_claim}
\bb P \pb{ \absb{ \norm{H}  - \bb E \norm{H}  }  \le   \del}+ \bb P \pb{  \norm{H}   \le 2+ 2 C \del}\;>\; 1\,,
\end{equation}
where $C$ is the constant in Proposition \ref{prop:norm_H0} and
\begin{equation*}
\delta \deq  \frac{ K \eta}{ \sqrt{1 \vee  \log \eta}} \,, \qquad \eta \deq \frac{\sqrt{\log n}}{q}\,.
\end{equation*}

Supposing for now that \eqref{prob_claim} has been proved, we find that the intersection of both events from the left-hand side of \eqref{prob_claim} is nonempty, and hence $\bb E\norm{H}   \le 2+ (2 C + 1)\del$, which concludes the proof of \eqref{est_EH}.

What remains, therefore, is to prove \eqref{prob_claim}. We first remark that, with $c_0 = 2 /  \me$, 
\begin{equation}\label{eq:lbqd}
q \delta  =  \frac{ K \sqrt {\log n}}{ \sqrt {1 \vee \log \eta} }  =\frac{ K \eta q}{ \sqrt {1 \vee \log \eta} }   \geq c_0 K 
\end{equation}
 uniformly over all $n \geq 2$ and $q \geq 1$ (distinguish $\eta \leq \me$ and $\eta \geq \me$). We shall assume that $c_0 K \geq 1$.
By \eqref{eq:BLM} and \eqref{eq:propo4.9}, it suffices to prove that, by choosing $K$ large enough,  the three numbers 
\begin{equation}\label{eq:1791611} a_1 \deq q^2\del^2\,, \qquad a_2 \deq q\log(1+\sqrt{\del})\,, \qquad a_3 \deq q^2h(\del \vee \delta^2)-\log n\ee
can be made arbitrarily large, uniformly under the conditions \eqref{eq:179161}. Indeed, the term $C/q$ from the left-hand side of \eqref{eq:propo4.9} is bounded by $C \delta / ( c_0 K)  \leq C \delta$ by \eqref{eq:lbqd}.

First, from \eqref{eq:lbqd}, $a_1 \geq  ( c_0 K )^2$. Hence, $a_1$ can be chosen arbitrarily large if $K$ is large enough. Similarly, to prove that $a_2$ can be chosen arbitrarily large, we note that the function $f(x) = \log(1 + \sqrt x)/x$ is positive and decreasing on $(0,1)$. Hence, from \eqref{eq:lbqd}, we have
\begin{equation*}
a_2 \geq   (c_0 K ) f(c_0 K/q ) \geq ( c_0 K ) f( c_0 K  ) = \log(1 + \sqrt{c_0K})\,,
\end{equation*}
as desired.

To prove that $a_3$ can be chosen arbitrarily large, we consider the cases $\eta \leq \me$ and $\eta \geq \me$ separately. For $\eta \leq \me$, we use $h(x) \geq c(x^2\wedge x)$ for all $x \geq 0$ and for some universal constant $c>0$. Hence $h( \delta \vee \delta^2) \geq c \delta^2$ and
\begin{equation*}
a_3 \geq c q^2 \delta^2 - \log n = (c K^2 - 1) \log n\,,
\end{equation*}
as desired. For $\eta \geq \me$,  we use the refined bound $h(x) \geq c (x^2\wedge x ) ( 1 \vee \log   x ) $ for all $x \geq 0$. Hence $h( \delta \vee \delta^2) \geq c \delta^2  \log \delta^2$ and, since $\log \delta^2 \geq c' \log \eta$  for some constant $c'>0$,  
\begin{equation*}
a_3 \geq c c' q^2 \delta^2 \log \eta  - \log n = (c c' K^2 - 1) \log n\,,
\end{equation*}
as desired. This concludes the proof of \eqref{est_EH}.

Finally, we prove \eqref{eq:mainH}. First,  Jensen's inequality implies that $\bb E \norm{H}_{2\to\infty} \geq 1$. Note also the triangle inequality gives $| \norm{X}_{2\to\infty} - \norm{Y}_{2\to\infty} | \leq \norm{X - Y}_{2\to\infty} \leq \sqrt{ \sum_{ij} |X_{ij} - Y_{ij} |^2}$, and the function $X \mapsto \norm{X}_{2 \to \infty}$ is separately convex in the entries of $X$ (as a maximum \eqref{def_norms_H} of separately convex functions). Hence, we may apply Talagrand's concentration inequality in the form of \cite[Theorem 6.10]{BLM}. We find that there exists a universal constant $c > 0$ such that for any $t > 0$,  
\begin{equation}\label{eq:BLM2}
\bb P \PAR{ { \absb{ \norm{H}_{2\to\infty}  - \bb E \norm{H}_{2\to\infty}  }  \geq t  / q}} \leq 2 e^{-ct ^2}\,. 
\end{equation}
This last inequality and \eqref{eq:BLM} implies that if we prove that the event 
 \begin{equation}\label{eq:mainH0}
\hbb{ \norm{H}  \leq  \norm{H}_{2\to\infty} \PAR{ 2 +  \frac C q }}
 \end{equation} 
  has probability at least $1/2$, then 
  $$
  \bb E  \norm{H}  \leq \bb E \norm{H}_{2\to\infty} \PAR{ 2 +  \frac C q } + \frac{2t}{q} \leq \bb E \norm{H}_{2\to\infty} \PAR{ 2 +  \frac{ C + 2t} q }\,,
  $$
  where $t$ is such that $4 e^{-ct^2} < 1/2$, the constant $c$ is as in  \eqref{eq:BLM}--\eqref{eq:BLM2}, and we used that $\bb E \norm{H}_{2\to\infty} \geq 1$. Therefore, it suffice to prove that \eqref{eq:mainH0} holds with probability at least $1/2$. With probability at least $3/4$, for some universal constant $t>0$, we have 
  $$
    \rho(B) \leq 1 + \frac{t}{q}\,. 
  $$ 
  and, from \eqref{eq:BLM2} with probability at least $3/4$, 
  $$
  \norm{H}_{2\to\infty} \geq \bb E \norm{H}_{2\to\infty} - \frac{t}{q} \geq 1 - \frac{t}{q}\,.
  $$
  We deduce from Corollary \ref{cor:From_NBmatrix_to_matrix} that with probability at least $1/2$, if $q \geq 2 t$, 
  $$
\bb  E  \norm{H} \leq 2 \norm{H}_{2\to\infty} + C \PAR{ \frac{( 2t / q) ^2}{ 1 - t / q} + \frac{ 1}{q}} \leq 2 \norm{H}_{2\to\infty} + \frac{C'}{q} \leq 2 \norm{H}_{2\to\infty} + \frac{\norm{H}_{2\to\infty}}{1 - C / q} \frac{ C'}{q}\,.  
  $$
  Adjusting the constant $C$, we obtain \eqref{eq:mainH0}. This concludes the proof of \eqref{eq:mainH}, and hence also of Theorem \ref{thm:norm_H}.

\section{Non-Hermitian matrices: proof of Theorem \ref{thm:norm_hH}} \label{sec:6}

The main estimate of this section is the following result.

\begin{proposition} \label{prop:nH_moment_est}
Let $H \in M_n(\C)$ be a random matrix whose entries $(H_{ij})_{1 \leq i , j \leq n}$ are independent mean-zero random variables. Moreover, suppose that there exist $q \geq 3$ and $\kappa \geq 1$ such that \eqref{H_moment_ass} holds. There exist universal constants $c_0 , C_0 > 0$ such that the following holds.
If $\ell \geq 1$, $q \geq 1$, and $\delta \in (0,1/3)$ satisfy \eqref{ell_assumptions} then $\E \tr H^\ell H^{*\ell} \leq C_0 n \ell^{4} q^2$.
\end{proposition}

Once Proposition \ref{prop:nH_moment_est} is proved, Theorem \ref{thm:norm_hH} follows by the argument of Section \ref{sec:pfTh14}.

The rest of this section is devoted to the proof of Proposition \ref{prop:nH_moment_est}. The proof of Proposition \ref{prop:nH_moment_est} is similar to that of Proposition \ref{prop:BB_bound}. Essentially, the nonbacktracking condition in the definition of $B$ is replaced by the independence of the entries of $H$. We compute
\begin{equation*}
\Tr H^\ell H^{* \ell} = \sum_{\xi} H_{\xi^1_0 \xi^1_1 } H_{\xi^1_1 \xi^1_2} \cdots H_{\xi^1_{\ell - 1} \xi^1_\ell} \ol H_{\xi^2_0 \xi^2_1 } \ol H_{\xi^2_1 \xi^2_2} \cdots \ol H_{\xi^2_{\ell - 1} \xi^2_\ell}\,,
\end{equation*}
where the sum ranges over all $\xi = (\xi^\nu_i \col i \in \{0,1, \dots, \ell\}, \nu \in \{1,2\}) \in [n]^{2\ell + 2}$ such that $(\xi^1_0, \xi^1_\ell) = (\xi^2_0, \xi^2_\ell)$.

Because the entries of $H$ are independent and have mean zero, we find
\begin{equation} \label{est_nH1}
\E \Tr H^\ell H^{* \ell} = \E \sum_{\xi \in \cal C} H_{\xi^1_0 \xi^1_1 } H_{\xi^1_1 \xi^1_2} \cdots H_{\xi^1_{\ell - 1} \xi^1_\ell} \ol H_{\xi^2_0 \xi^2_1 } \ol H_{\xi^2_1 \xi^2_2} \cdots \ol H_{\xi^2_{\ell - 1} \xi^2_\ell}\,,
\end{equation}
where $\cal C$ is the set of pairs $\xi = (\xi^1, \xi^2) \in [n]^{2\ell + 2}$ satisfying $(\xi^1_0, \xi^1_\ell) = (\xi^2_0, \xi^2_\ell)$ and
\begin{equation*}
\sum_{\nu = 1}^2 \sum_{i = 1}^{\ell - 1} \ind{(\xi^\nu_{i - 1}, \xi^\nu_i) = (a,b)} \neq 1
\end{equation*}
for all $a,b \in [n]$.

As in Section \ref{sec:pfpropBB}, we estimate the right-hand side of \eqref{est_nH1} using graphs. In contrast to Section \ref{sec:pfpropBB}, in this section we always use directed graphs and multigraphs. The following definitions closely mirror those from Section \ref{sec:pfpropBB}. By definition, a (vertex-labelled) \emph{directed multigraph} $G = (V(G), E(G), \phi)$ consists of two finite sets, the set of \emph{vertices} $V(G)$ and the set of \emph{edges} $E(G)$, and a map $\phi = (\phi_+, \phi_-)$ from $E(G)$ to the ordered pairs of elements of $V(G)$. The edge $e \in E(G)$ is a \emph{loop} if $\phi_+(e) = \phi_-(e)$. We define the \emph{outdegree} $\deg_+(v) \deq \sum_{e \in E(G)} \ind{\phi_+(e) = v}$, the \emph{indegree} $\deg_-(v) \deq \sum_{e \in E(G)} \ind{\phi_-(e) = v}$, and the \emph{degree} $\deg(v) \deq \deg_+(v) + \deg_-(v)$ of a vertex $v \in V(G)$. As always, the \emph{genus} of $G$ is $g(G) \deq \abs{E(G)} - \abs{V(G)} + 1$. A \emph{path of length $l \geq 1$ in $G$} is a word $w = w_0 w_{01} w_1 w_{12} \dots w_{l-1 l} w_l$ such that $w_0, w_1, \dots, w_l \in V(G)$, $w_{01}, w_{12}, \dots, w_{l-1 l} \in E(G)$, and $\phi(w_{i - 1 i}) = (w_{i - 1}, w_i)$ for $i = 1, \dots, l$. We denote the length $l$ of $w$ by $\abs{w} \deq l$.
For $e \in E(G)$ we define the \emph{number of crossings of $e$ by $w$} to be $m_e(w) \deq \sum_{i = 1}^l \ind{w_{i-1 i} = e}$.

Moreover, the directed multigraph $G$ is called simply a \emph{directed graph} if $\phi$ is injective, i.e.\ there are no multiple edges. (Note that in our convention a directed graph may have loops.) For a directed graph $G$ we may and shall identify $E(G)$ with a set of ordered pairs of $V(G)$, simply identifying $e$ and $\phi(e)$. Similarly,  we identify a path $w$ with the reduced word $w_0 w_1 \dots w_l$ only containing the vertices, since we must have $w_{i - 1 i} = (w_{i - 1}, w_i)$ for $i = 1, \dots, l$.

The following definition is analogous to Definition \ref{def_Gxi}.
\begin{Def}
For $\xi \in \cal C$ we define the directed graph $G_\xi$ as
\begin{equation*}
V(G_\xi) \deq \hb{\xi^\nu_i \col i \in \{0, \dots, \ell\}, \nu \in \{1,2\}} \,, \qquad E(G_\xi) \deq \hb{(\xi^\nu_{i - 1}, \xi^\nu_i) \col i \in \{1, \dots, \ell\}, \nu \in \{1,2\}}\,.
\end{equation*}
Thus, $\xi^1 \equiv \xi^1_0 \xi^1_1 \cdots \xi^1_{\ell}$ and $\xi^2 \equiv \xi^2_0 \xi^2_1 \cdots \xi^2_{\ell}$ are paths in $G_\xi$.
\end{Def}

As in Section \ref{sec:pfpropBB}, we introduce an equivalence relation on $\cal C$ by saying that $\xi, \tilde \xi \in \cal C$ are equivalent if and only if there exists a permutation $\tau$ of $[n]$ such that $\tau(\xi_i^\nu) = \tilde \xi_i^\nu$ for all $i = 0, \dots, 2 \ell$ and $\nu = 1,2$.
The following result is analogous to Lemma \ref{lem:sum_gamma_Gamma}.

\beg{lem} \label{lem:sum_gamma_Gamma_nH}
For any $\bar \xi \in \cal C$ we have
$$\E\sum_{\xi \in [\bar \xi]}  H_{\xi^1_0 \xi^1_1 } H_{\xi^1_1 \xi^1_2} \cdots H_{\xi^1_{\ell - 1} \xi^1_\ell} \ol H_{\xi^2_0 \xi^2_1 } \ol H_{\xi^2_1 \xi^2_2} \cdots \ol H_{\xi^2_{\ell - 1} \xi^2_\ell} \le n^{1-g(G_{\bar \xi})} \kappa^{g(G_{\bar \xi})} q^{2\abs{E(G_{\bar \xi})} - 2\ell}\,.$$
\en{lem} 

\begin{proof}
The proof is almost identical to that of Lemma \ref{lem:sum_gamma_Gamma}. Using its notations, we suppose without loss of generality that $V(G_{\bar \xi}) = [s]$ and pick a (directed) spanning tree $T$ of $G_\xi$ such that, for all $t = 2, \dots, s$ the subgraph $T\vert_{[t]}$ is a spanning tree of $[t]$, and enumerate the edges $e_1, \dots, e_{s - 1}$ of $T$ so that for all $t = 1, \dots, s-2$, the edge $e_t$ is an edge of $T\vert_{[t+1]}$ that is not an edge of $T \vert_{[s]}$ and $\phi_-(e_t) = t$. Then the argument from the proof of Lemma \ref{lem:sum_gamma_Gamma} carries over verbatim, using the estimates \eqref{Hyp0MomentsH} and \eqref{HypMomentsH}, which also hold under the assumptions of Proposition \ref{prop:nH_moment_est}.
\end{proof}

Next, we define normal paths in analogy to Definition \ref{def:normal}.

\begin{Def} \label{def:normal_nh}
For $\nu = 1,2$, let $w^\nu = w^\nu_0 w^\nu_{01} w^\nu_1 w^\nu_{12} \dots w^\nu_{l^\nu-1 l^\nu} w^\nu_{l^\nu}$ be a path in a directed multigraph $G$. We say that $(w^1, w^2)$ is \emph{normal in $G$} if
\begin{enumerate}
\item
$V(G) = \{w^1_0, \dots, w^1_{l^1}, w^2_0, \dots, w^2_{l^2}\} = [s]$ where $s \deq \abs{V(G)}$;
\item
the vertices of $V(G)$ are visited in increasing order first by $w^1$ and then by $w^2$, i.e.\ if $w^1_i \notin \{w^1_0, \dots, w^1_{i - 1}\}$ then $w^1_i > w^1_1, \dots, w^1_{i - 1}$, and if $w^2_i \notin \{w^1_0, \dots, w^1_{l^1}, w^2_0, \dots, w^2_{i - 1}\}$ then $w^1_i > w^1_0, \dots, w^1_{l^1}, w^2_0, \dots, w^2_{i - 1}$.
\end{enumerate}
\end{Def}
Each equivalence class of $\sim$ in $\cal C$ has a unique representative $\xi = (\xi^1,\xi^2)$ that is normal in $G_\xi$. We denote $\cal C_0 \deq \h{\xi \in \cal C \col \xi \text{ normal in } G_\xi}$. Thus, from \eqref{est_nH1} and Lemma \ref{lem:sum_gamma_Gamma_nH} we deduce that
\begin{equation} \label{BB_est1_nH}
\E\Tr H^\ell{H^{\ell *}} \le n \sum_{\xi \in \cal C_0} n^{-g(G_\xi)}\kappa^{g(G_\xi)} q^{2\abs{E(G_\xi)} - 2\ell}\,.
\end{equation}

\begin{figure}[!ht]
\begin{center}
{\scriptsize 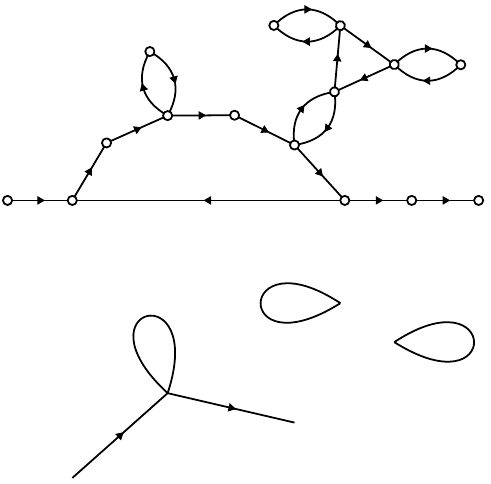}
\end{center}
\caption{In the top diagram we draw the graph $G_\xi$ associated with the pair $\xi = (\xi^1, \xi^2)$, where $\xi^1 =$ 1,2,3,4,5,6,7,8,9,8,9,8,10,11,10,7,8,10,11,10,7,6,7,6,12,13,14 and $\xi^2 = $ 1,2,3,4,15,4,5,6,12,2,3,4,15,4,15,4,5,6,12,2,3,4,5,6,12,13,14. Here $\ell = 26$. The total number of crossings by $\xi^1$ and $\xi^2$ of the edges $(2,3)$, $(3,4)$, $(4,5)$, $(5,6)$, $(6,12)$, $(12,2)$, $(4,15)$, $(15,4)$ is three, and of all other edges two. Note that $\xi$ is normal in $G_\xi$.
In the bottom diagram we draw the directed multigraph $U$ associated with $\xi$, which in this example is just a directed graph (i.e.\ it has no multiple edges). The paths $\zeta^1, \zeta^2$ in $U$ associated with $\xi$ are $\zeta^1 = $ 12345666775677545489 and $\zeta^2 = $ 12334823334823489. The pair $(\zeta^1, \zeta^2)$ is also normal in $U$. Here $\gamma = 9$.
\label{fig:U_directed}}
\end{figure}

We now introduce a parametrization $(U, \zeta, k)$ of $\cal C_0$ obtained by deleting vertices of $G_\xi$ that have degree two, except the vertices $1 = \xi_0^1 = \xi_0^2$ and $\xi_\ell^1 = \xi_\ell^2$. The construction follows verbatim that of Section \ref{sec:pfpropBB}, whereby all graphs and multigraphs are directed. Note that every vertex of the directed graph $G_\xi$ that is not $1$ or $\xi^1_\ell = \xi^2_\ell$ and has degree two has indegree one and outdegree one. More formally, we define $\cal I(\xi) \deq \{v \in V(G_\xi) \setminus \{1, \xi^1_\ell = \xi^2_\ell\}\}$ as well as $\Sigma(\xi)$ to be the set of directed paths $w = w_0 \dots w_l$ in $G_\xi$ such that $w_1, \dots, w_{l - 1} \in \cal I(\xi)$ and $w_0, w_l \in \cal I(\xi)$. Then we define $\hat G_\xi$ to be the directed multigraph obtained from $G_\xi$ by replacing each directed path $w \in \Sigma(\xi)$ by a directed edge from $w_0$ to $w_l$. We denote the resulting paths in $\hat G_\xi$ associated with $\xi^1, \xi^2$ by $\hat \xi^1, \hat \xi^2$, and the length of the path $w$ associated with $e \in E(\hat G_\xi)$ by $\hat k_e$. Applying a suitable bijection $\tau \col V(\hat G_\xi) \to [\abs{V(\hat G_\xi)}]$, we obtain the triple $(U,\zeta,k)$, where $U$ is a directed multigraph, $\zeta = (\zeta^1, \zeta^2)$ is a pair of paths normal in $U$ of lengths $r^1,r^2$ satisfying $1 = \zeta^1_0 = \zeta^2_0$ and $\gamma \deq \zeta^1_{r^1} = \zeta^2_{r^2}$, and $k = (k_e)_{e \in E(U)}$ is the family of weights of the edges of $U$. See Figure \ref{fig:U_directed} for an illustration of the construction of $(U,\zeta,k)$. As in Section \ref{sec:pfpropBB} every vertex in $V(U) \setminus \{1, \gamma\}$ has degree at least three. The remainder of the proof now follows to the letter the argument from Section \ref{sec:pfpropBB} starting on page \pageref{U_prop}. This concludes the proof of Proposition \ref{prop:nH_moment_est}.

\bibliographystyle{abbrv}
\bibliography{bib}

\bigskip

\noindent
Florent Benaych-Georges, 
MAP 5 (CNRS UMR 8145) - Universit\'e Paris Descartes, 45 rue des Saints-P\`eres 75270 Paris cedex~6,  France. Email:
\href{mailto:florent.benaych-georges@parisdescartes.fr}{florent.benaych-georges@parisdescartes.fr}.
\\[1em]
Charles Bordenave, 
Institut de Math\'ematiques de Marseille (CNRS UMR 7373) - Aix-Marseille Universit\'e, 163 Avenue de Luminy  13009 Marseille. France. Email:
\href{mailto:charles.bordenave@univ-amu.fr}{charles.bordenave@univ-amu.fr}.
\\[1em]
Antti Knowles, 
University of Geneva, Section of Mathematics, 2-4 Rue du Li\`evre, 1211 Gen\`eve 4, Switzerland. 
Email:
\href{mailto:antti.knowles@unige.ch}{antti.knowles@unige.ch}.

\en{document}